\documentclass[12pt]{amsart}

\usepackage[margin=3cm]{geometry}

\title{Refined Bounds on Near Optimality Finite Window Policies in POMDPs and Their Reinforcement Learning}

\author{Yunus Emre Demirci}
\address{Queen's University, Department of Mathematics and Statistics, Kingston, Ontario, Canada}
\email{21yed@queensu.ca}

\author{Ali Devran Kara}
\address{Florida State University, Department of Mathematics}
\email{alikara@umich.edu}

\author{Serdar Yüksel}
\address{Queen's University, Department of Mathematics and Statistics, Kingston, Ontario, Canada} 
\email{yuksel@queensu.ca}

\newcommand{\su}{\textnormal{supp }}
\newcommand{\bea}{\begin{eqnarray}}
\newcommand{\ena}{\end{eqnarray}}
\newcommand{\beas}{\begin{eqnarray*}}
\newcommand{\enas}{\end{eqnarray*}}
\newcommand{\beq}{\begin{equation}}
\newcommand{\enq}{\end{equation}}

\newcommand{\Lip}{\operatorname{Lip}}

\newcommand\norm[1]{\left\lVert#1\right\rVert}

\newcommand{\T}{\mathcal{T}}
\newcommand{\X}{\mathbb{X}}
\newcommand{\Y}{\mathbb{Y}}
\newcommand{\U}{\mathbb{U}}

\usepackage{mathptmx} 

\usepackage{times} 

\usepackage{amsmath} 

\usepackage{amssymb}  

\usepackage{dsfont}

\usepackage{mathtools}

\usepackage[utf8]{inputenc}

\usepackage{graphicx}

\usepackage{newtxtext, newtxmath}

\usepackage{verbatim}

\usepackage{amsfonts}

\usepackage{enumitem}

\usepackage{bbm}

\usepackage{xcolor}

\newcommand{\1}{\mathbbm{1}}

\newcommand{\N}{\mathbb{N}}

\let\phi\varphi

\newcommand{\Z}{\mathcal{Z}}

\newcommand{\B}{\mathcal{B}}

\newcommand{\sT}{\mathcal{T}}

\newcommand{\Zplus}{\mathbb{Z}_{\geq 0}}

\DeclareMathOperator*{\argmin}{arg\,min}

\allowdisplaybreaks

\newtheorem{remark}{Remark}

\newtheorem{example}{Example}
\newtheorem{assumption}{Assumption}
\newtheorem{definition}{Definition}
\newtheorem{theorem}{Theorem}
\newtheorem{lemma}{Lemma}
\newtheorem{corollary}{Corollary}


\begin{document}


\maketitle

\begin{abstract}   
Finding optimal policies for Partially Observable Markov Decision Processes (POMDPs) is challenging due to their uncountable state spaces when transformed into fully observable Markov Decision Processes (MDPs) using belief states. Traditional methods such as dynamic programming or policy iteration are difficult to apply in this context, necessitating the use of approximation methods on belief states or other techniques. Recently, in \cite{kara2020near} and \cite{kara2021convergence}, it was shown that sliding finite window based policies are near-optimal for POMDPs with standard Borel valued hidden state spaces, and can be learned via reinforcement learning, with error bounds explicitly dependent on a uniform filter stability term involving total variation in expectation and sample path-wise, respectively. In this paper, we refine these performance bounds and (i) extend them to bounds via uniform filter stability in expected Wasserstein distance leading to an error bound in expectation, and (ii) complementary conditions bounds via uniform filter stability in sample path-wise total variation distance leading to a uniform error bound. We present explicit examples. Our approach thus provides complementary and more refined bounds on the error terms in both total variation and Wasserstein metrics, offering more relaxed and stronger bounds over the approximation error in POMDP solutions on the performance and near optimality of sliding finite window control policies.
\end{abstract}

\section{Introduction}
While practically very relevant, the study of partially observable Markov decision processes (POMDPs) is mathematically challenging.

Let $\X$ denote a standard Borel space (that is, a Borel subset of a complete, separable, and metric space), which is the state space of a partially 
observed controlled Markov process. 
Let $\mathcal{B}(\mathbb{X})$ be its Borel 
$\sigma$-field.
Let $\mathbb{C}_b(\mathbb{X})$ be the set of all 
continuous, bounded functions on $\mathbb{X}$. 
Here and throughout the 
paper $\mathbb{Z}_+$ denotes the set of non-negative
integers and $\N$ 
denotes the set of positive integers. Let
$\Y$ be a standard Borel space denoting the 
observation space of the model, 
and let the state be observed through an
observation channel $Q$. 
The observation channel, $Q$, is 
defined as a stochastic kernel (regular
conditional probability) from  
$\X$ to $\Y$, such that
$Q(\,\cdot\,|x)$ is a 
probability measure on the power set 
$P(\Y)$ of $\Y$ for every $x
\in \mathbb{X}$, and $Q(A|\,\cdot\,): \X \to [0,1]$ 
is a Borel
measurable function for every $A \in P(\Y)$.  A
decision maker (DM) is located at the output 
of the channel $Q$, and hence only sees 
the observations $\{Y_t,\, t\in \Zplus\}$ and chooses its actions 
from $\U$, the action space, which is assumed to be a compact subset of a Polish metric space. 
An {\em admissible policy} $\gamma$ is a
sequence of control functions $\{\gamma_t,\, t\in \Zplus\}$ such
that $\gamma_t$ is measurable with respect to the $\sigma$-algebra
generated by the information variables
$
I_t=\{Y_{[0,t]},U_{[0,t-1]}\}, \quad t \in \mathbb{N}, \quad
I_0=\{Y_0\},
$
where
\begin{equation}
\label{eq_control}
U_t=\gamma_t(I_t),\quad t\in \Zplus,
\end{equation}
are the $\mathbb{U}$-valued control
actions and 
$Y_{[0,t]} = \{Y_s,\, 0 \leq s \leq t \}, \quad U_{[0,t-1]} =
\{U_s, \, 0 \leq s \leq t-1 \}.$
In the above, the dependence of control policies on the 
initial distribution $\mu$ is implicit. We will 
denote the collection of admissible control policies as $\Gamma$.
The update rules of the system are 
determined by (\ref{eq_control}) and the following
relationships:
\[  \Pr\bigl( (X_0,Y_0)\in B \bigr) =  \int_B \mu(dx_0)Q(dy_0|x_0), \quad B\in \mathcal{B}(\mathbb{X}\times\mathbb{Y}), \]
where $\mu$ is the (prior) distribution of the initial state $X_0$, and
\begin{eqnarray*}
\label{eq_evol}
 &\Pr\biggl( (X_t,Y_t)\in B \, \bigg|\, (X,Y,U)_{[0,t-1]}=(x,y,u)_{[0,t-1]} \biggr)  = \int_B Q(dy_t|x_t)   \mathcal{T}(dx_t|x_{t-1}, u_{t-1}),  
\end{eqnarray*}
$B\in \mathcal{B}(\mathbb{X}\times\mathbb{Y}), t\in \mathbb{N},$ 
where $\sT$ is the transition 
kernel of the model which is a stochastic 
kernel from $\X\times\U$ to $\mathbb{X}$. 
We will consider the following expected discounted cost criterion
$$
J_\beta(\mu, \gamma)=E_\mu^{\gamma}\left[\sum_{t=0}^{\infty} \beta^t c\left(X_t, U_t\right)\right]
$$
for some discount factor $\beta \in(0,1)$, over the set of admissible policies $\gamma \in \Gamma$, where $c: \mathbb{X} \times \mathbb{U} \rightarrow \mathbb{R}$
is the stage-wise measurable cost function, which will be assumed to be continuous and bounded. Here, the expectation $E_\mu^{\gamma}$ is taken over the initial state probability measure $\mu$ under policy $\gamma$.
The optimal cost for the discounted infinite horizon is defined as 
$$J_\beta^*(\mu)=\inf _{\gamma \in \Gamma} J_\beta(\mu, \gamma).$$


In the following, we provide a literature review followed by a summary of our main results. 

\section{Literature Review and Contributions}

The existence and derivation of optimal policies often involve converting the problem into a fully observable MDP using “belief states” and then applying classic methods like dynamic programming or policy iteration. However, since the resulting fully observable MDP has an uncountable state space, approximation methods on belief states or other approximation techniques are required.

There have been many studies on approximation methods for POMDPs: \cite{porta2006point} developed computational algorithms that exploit the structural convexity/concavity properties of the value function under the discounted cost criterion. \cite{spaan2005perseus} proposed an algorithm that essentially quantizes the belief space, although they did not provide rigorous convergence results. \cite{smith2012point} and \cite{pineau2006anytime} also presented quantization-based algorithms for belief states, assuming finite state, measurement, and action sets. \cite{zhou2008density}, \cite{zhou2010solving} conducted another rigorous study, providing an explicit quantization method for the set of probability measures containing belief states, under the assumption that the state space is parametrically representable with strong density regularity conditions.

Studies on on the regularity properties of POMDPs, as introduced in \cite{CrDo02} and \cite{FeKaZg14}, and later extended to broader conditions and criteria in \cite{KSYWeakFellerSysCont}, \cite{feinberg2022markov}, \cite{Feinberg2023}, \cite{kara2020near}, \cite{demirci2023geometric}, have enabled explicit performance bounds in part thanks to rigorous approximation results for weak Feller MDPs in standard Borel spaces \cite{SaYuLi15c} \cite{SYLTAC2017POMDP} . In view of these results, different from filter stability based methods to be presented in this paper, the least stringent conditions on near optimal finite approximation results for POMDPs, to our knowledge, were presented \cite{SYLTAC2017POMDP}, building on \cite{SaYuLi15c}, which only required weak continuity conditions on the reduced MDP. This study conducted rigorous approximation analysis and developed explicit methods for quantizing probability measures after establishing weak continuity conditions on the belief MDP. They demonstrated that finite model approximations obtained through quantization are asymptotically optimal, and the control policies derived from the finite model can be applied to the actual system with diminishing error as the number of quantization bins increases. For detailed reviews on approximate optimality in POMDPs, we refer to \cite{SYLTAC2017POMDP} for discounted cost problems and \cite{DeKaYu2024} for average cost problems.

Our paper generalizes the recent analysis introduced in \cite{kara2021convergence} and \cite{kara2020near} which explicitly relate approximation bounds on the use of sliding finite window based policies to filter stability. 

In particular, \cite{kara2021convergence} proposes a finite history approximation method. This method uses Q-learning with finite history to achieve a near-optimal policy, with the error related to controlled expected uniform filter stability in terms of total variation. We denote this filter stability term with:
\begin{align}\label{def_L_t}
L_t^N:=\sup _{\hat{\gamma} \in \hat{\Gamma}} E_{z_0^{-}}^{\hat{\gamma}}\left[ \right. \left\| \right. P^{z_t^{-}}\left(X_{t+N} \in \cdot \mid Y_{[t, t+N]}, U_{[t, t+N-1]}\right)-
 P^{z^*}\left(X_{t+N} \in \cdot \mid Y_{[t, t+N]}, U_{[t, t+N-1]}\right)  \left.\right\|_{T V} \left.\right]
\end{align}
    Here, $\hat{\gamma} \in \hat{\Gamma}$, where $\hat{\Gamma}$ 
    can be taken to be those Markov control policies under 
    controlled states given with $\hat{z}_n$ defined in (\ref{zhat}); 
    that is, policies which map 
    $\left(\hat{z}_n, n\right) \rightarrow$ $u_n$ for all 
    $n \in \mathbb{Z}_{+}$.
     \cite[Theorem 3]{kara2021convergence} showed that with $\hat{z}_0=(z_0^-,I_0^N)$, with a policy $\hat{\gamma}$ acting on the first $N$ steps
\begin{align}\label{Kara_yuksel}
E_{z_0^-}^{\hat{\gamma}}\left[\left|J_\beta(\hat{z}_0,\tilde{\phi}^N) -J^*_\beta(\hat{z}_0)\right||I_0^N\right]\leq  \frac{2\|c\|_\infty }{(1-\beta)}\sum_{t=0}^\infty\beta^t L^N_t,
\end{align}
where $\tilde{\phi}^N$ is the optimal finite window policy, explicitly defined in equation (\ref{optimalfinite}).

       
\cite[Section 4.2 and Theorem 17]{kara2020near} presented the following alternative condition involving sample path-wise uniform filter stability term:
 \begin{align}\label{TVUnifB}
\bar{L}^N_{TV}:=\sup_{z\in \mathcal{P}(\mathds{X})}\sup_{y_{[0,N]},u_{[0,N-1]}} \left\|P^{z}(\cdot|y_{[0,N]},u_{[0,N-1]})-P^{z^*}(\cdot|y_{[0,N]},u_{[0,N-1]})\right\|_{TV},
\end{align}
and showed that
\begin{align}\label{jmlrboundF}
\sup_z\left|J_\beta(z,\gamma_N)-J^*_\beta(z)\right|\leq \frac{2(1+(\alpha_{\mathcal{Z}}-1)\beta)}{(1-\beta)^3(1-\alpha_{\mathcal{Z}}\beta)}\|c\|_\infty \bar{L}^N_{TV}
\end{align}
for all $\beta \in (0,1)$ under a contraction condition, for some constant $\alpha_{\mathcal{Z}}$ defined in \cite{kara2020near}.  However, the analysis did not provide explicit conditions for $\bar{L}^N_{TV}$ to converge to zero as $N$ increases to $\infty$. Additionally, \cite[Theorem 9]{kara2020near} provided conditions where the error is in the bounded-Lipschitz metric (which is equivalent to the Wasserstein-1 metric when the state space $\mathbb{X}$ is compact), however these were only applicable for a restrictive subset of the discount parameter $\beta$. On the other hand, the bound in (\ref{Kara_yuksel}) is in expectation whereas the bound in (\ref{jmlrboundF}) is uniform, and thus the results are complementary.
  
In this paper, we provide refined bounds for the error terms in the Wasserstein-1 metric unlike \cite{kara2021convergence}, and which are applicable to all $\beta \in (0,1)$ unlike \cite{kara2020near}. Furthermore, we present explicit conditions for $\bar{L}^N_{TV}$, which was not studied in \cite{kara2020near}.

For reinforcement learning, significant studies are documented in the literature, with \cite{singh1994learning} being a significant earlier contribution. \cite{kara2021convergence} showed that viewing the finite memory as an artificial state allows for convergence of a Q-learning algorithm, where convergence is to an approximate finite state MDP, which is then near optimal under filter stability conditions; see the general review in \cite{karayukselNonMarkovian}. \cite{white1994finite} investigates approximation techniques for POMDPs using finite memory, focusing on finite state, action, and measurement sets. They reduce the POMDP to a belief MDP and utilize worst and best case predictors based on the N most recent information variables to construct an approximate belief MDP. 
\cite{yu2008near} examines the near optimality of finite window policies for average cost problems involving finite state, action, and observation spaces, and demonstrates that, under the condition that the average optimal cost under the liminf criterion is independent of the initial state, finite memory policies are near optimal and the limsup and liminf of the average cost are equal; the proof method relies on convex analysis and crucially requiring finiteness of the state space. \cite{cayci2022} provides a performance bound related to the window length by considering persistence of excitation of the optimal policy and minorization-majorization assumptions, which will be discussed further in the paper. We also note the recent findings in \cite{KSYContQLearning} and \cite{karayukselNonMarkovian}, which provides a general framework for non-Markovian models. For complementary approaches to learning problems, we refer to the paper \cite{SinhaMahajan2024}. \\

{\bf Contributions.}
\begin{itemize}
\item[(i)] [Bounds via uniform Filter Stability in Expected Wasserstein Distance] Relaxing the analysis in \cite{kara2021convergence}, we demonstrate that policies using only finite history are near-optimal. 
We derive an explicit upper bound for the error using the newly defined controlled filter stability term $\bar{L}_t^N$, which is expressed as:
    \[
    \bar{L}_t^N = \sup_{\hat{\gamma} \in \hat{\Gamma}} E_{z_0^{-}}^{\hat{\gamma}}\left[W_1\left(P^{z_t^{-}}\left(X_{t+N} \in \cdot \mid Y_{[t, t+N]}, U_{[t, t+N-1]}\right),P^{z^*}\left(X_{t+N} \in \cdot \mid Y_{[t, t+N]}, U_{[t, t+N-1]}\right)\right)\right]
    \]
and provide an explicit upper bound for it. Specifically, Theorem \ref{J_diff} demonstrates that under Assumption \ref{main_assumption}, the error decays exponentially with respect to the window size N, providing improved conditions for continuous state space systems while maintaining consistency for finite models. Our results are applicable for all $\beta \in (0,1)$ and relaxes total variation filter stability.
\item[(ii)] [Bounds via uniform Filter Stability in Sample Path-Wise Total Variation Distance] We obtain a uniform error bound over sample paths through a projective Hilbert metric \cite{le2004stability} based approach thus providing explicit and sufficient conditions for (\ref{jmlrboundF}). In Theorem \ref{d}, we show that under mixing conditions on the transition kernels (e.g., Assumption \ref{mixing_kernel_con} on the mixing coefficient $\epsilon_u$), the error in sample path-wise total variation distance ($\bar{L}^N_{TV}$), and consequently $L_t^N$, decays geometrically with N. This method complements the construction in (i) \cite{kara2021convergence}. There are problems where the contraction condition to be presented may not hold, yet the Hilbert metric-based approach is applicable. Here also, our results are applicable for all $\beta \in (0,1)$ unlike the analysis in \cite{kara2020near}. 
\item[(ii)] Explicit examples are provided in the paper. We thus give refined upper bounds, and explicit sufficient conditions and examples, leading to complementary and more relaxed conditions given in the literature (notably in \cite{kara2020near} and \cite{kara2021convergence}), which are respectively related to controlled filter stability in terms of uniform sample path-wise total variation or bounded-Lipschitz norms, and the uniform filter stability in total variation in expectation ($\bar{L}_{TV}$ and $L_t^N$, respectively). 
\end{itemize}


\section{Notation and preliminaries}
\subsection{Belief MDP reduction for POMDPs}
It is known that any POMDP can be reduced to a completely observable Markov process \cite{Rhe74, Yus76}, whose states are the posterior state distributions or {\it belief}s of the observer; that is, the state at time $n$ is
\begin{align}
z_n(\,\cdot\,) := P^{\mu}(X_{n} \in \,\cdot\, | y_0,\ldots,y_n, u_0, \ldots, u_{n-1}) \in {\mathcal P}(\mathbb{X}),
\end{align}
where the initial state $X_0$ has a prior distribution $\mu \in \mathcal{P}(\mathbb{X})$.
We call this equivalent process the filter process. 
We denote by
$\Z:={\mathcal P}(\mathbb{X})$ 
the set of probability measures on 
$(\mathbb{X}, \mathcal{B}(\mathbb{X}))$ 
under the weak convergence topology, where, 
under this topology $\mathcal{Z}$ 
is also a standard Borel space, that is, $\Z=\mathcal{P}(\mathbb{X})$ is separable and completely metrizable under the weak convergence topology.  
The filter process has state space 
$\mathcal{Z}$ and 
action space $\mathbb{U}$.
Let $\mathcal{P}(\mathcal{Z})$ denote the 
probability measures on $\mathcal{Z}$, 
equipped with the weak convergence topology. 
$\Y$ and $\U$ are finite sets.

The transition probability $\eta$ of the filter 
process can be determined via the 
following equation 
\cite{Rhe74, Yus76, HernandezLermaMCP}: 
\begin{align}\label{conteta}
\eta(\cdot \mid z, u)=\int_{\mathbb{Y}} 1_{\{F(z, u, y) \in \cdot\}} P(d y \mid z, u),
\end{align}
where 
$$P(\cdot \mid z, u)=\operatorname{Pr}\left\{Y_{n+1} \in \cdot \mid Z_n=z, U_n=u\right\}$$ 
from $\mathcal{Z} \times \mathbb{U}$ to $\mathbb{Y}$ and 
$$
F(z, u, y):=\operatorname{Pr}\left\{X_{n+1} \in \cdot \mid Z_n=z, U_n=u, Y_{n+1}=y\right\}
$$
from $\mathcal{Z} \times \mathbb{U} \times \mathbb{Y}$ to $\mathcal{Z}$.  The one-stage cost function $\tilde{c}:\mathcal{Z} \times \mathbb{U} \rightarrow \mathbb{R}$ is a Borel measurable function and it is given by
\begin{align}\label{tildecost}
\tilde{c}(z, u):=\int_{\mathbb{X}} c(x, u) z(d x),
\end{align}
where $c: \mathbb{X} \times \mathbb{U} \rightarrow \mathbb{R}$ is the stage-wise cost function.

This way, we obtain a completely observable Markov decision 
process from the POMDP, with the components 
$(\mathcal{Z}, \mathbb{U}, \tilde{c}, \eta)$. 
The resulting MDP is often referred to as the belief-MDP.
For finite horizon problems and a large class of 
infinite horizon discounted cost problems, 
it is well established that belief states 
serve as a sufficient statistic for an optimal control policy 
\cite{Rhe74, Yus76, HernandezLermaMCP}.

\subsection{Alternative finite window belief-MDP reduction.}\label{Alt_finite_Mdp}
In this subsection, we will first present the alternative finite window reduction introduced in \cite[Section 3.1]{kara2021convergence}. Following this, we will define the approximation of this finite window MDP. 

Recall that
$z_n$ is the 
belief distribution defined as
\begin{align}
z_n(\cdot)=P^{\mu}(X_n \in \cdot \mid y_0, 
\ldots, y_n, u_0, \ldots, u_{n-1}) \in \mathcal{P}(\mathbb{X}).
\end{align}
Define $z^-_{n}$ is the predictor distribution at time $n$ before 
observing $Y_n$:
\begin{align} 
z_n^-(\cdot):=P^\mu(X_n \in \cdot \mid y_0, 
\ldots, y_{n-1}, u_0, \ldots, u_{n-1}) \in \mathcal{P}(\mathbb{X}),
\end{align} 
where the initial state $X_0$ has a prior distribution $\mu \in \mathcal{P}(\mathbb{X})$.

For any $N,n \geq 0$, we can determine $z_{n+N}$ as follows: 
\begin{align}
z_{n+N}=
P^{z_n^-}\left(X_{n+N} \in \cdot \mid y_n, 
\ldots, y_{n+N}, u_n, \ldots, u_{n+N-1}\right). 
\end{align}

We then consider an alternative finite window 
belief MDP reduction \cite{kara2021convergence}. 
Let us define the 
state variable at time $n \geq N$ as:
\begin{align}\label{zhat}
\hat{z}_n=\left(z_{n-N}^{-}, I_n^N\right),
\end{align}
where, for $N \geq 1$, the components are:
\begin{align}
z_{n-N}^{-} & =P^{\mu}\left(X_{n-N} \in \cdot \mid y_{n-N-1}, \ldots, y_0, u_{n-N-1}, \ldots, u_0\right), \\
I_n^N & =\left\{y_n, \ldots, y_{n-N}, u_{n-1}, \ldots, u_{n-N}\right\},
\end{align}
and for $N=0$, $I_n^N=y_n$ with the prior measure $\mu$ on $X_0$.
The state space is thus 
$\hat{\mathcal{Z}}=\Z 
\times \mathbb{Y}^{N+1} \times \mathbb{U}^N$.

The natural mapping between state spaces 
is defined by $\psi: \hat{\mathcal{Z}} \rightarrow \Z$, 
such that:
\begin{align}
\psi\left(\hat{z}_n\right)=
\psi\left(z_{n-N}^{-}, I_n^N\right) & =
P^{z_{n-N}^{-}}\left(X_n \in \cdot \mid y_{n-N}, 
\ldots, y_n, u_{n-N}, \ldots, u_{n-1}\right)=z_n
\end{align}
The new transition kernel and cost function are defined as:
\begin{align}
\hat{\eta}(\cdot \mid \hat{z}, u)=
\int_{\mathbb{Y}} \1_{\{(z_{n-N+1}^{-}, I_{n+1}^N) \in \cdot\}}
 \hat{P}(d y \mid \hat{z}, u),
\end{align}
where 
$$\hat{P}(\cdot \mid \hat{z}, u)=
\operatorname{Pr}\left\{Y_{n+1} \in \cdot 
\mid Z_{n-N}=z_{n-N}^{-}, I_n^N, U_n=u\right\}$$ 
from $\mathcal{Z} \times \mathbb{U}$ to $\mathbb{Y}$.
The cost function is:
\begin{align*} 
    & \hat{c}\left(\hat{z}_n, u_n\right)
    =\tilde{c}\left(\psi\left(z_{n-N}^{-}, I_n^N
    \right), u_n\right)=
    \\ & \int_{\mathbb{X}} c
    \left(x_n, u_n\right) P^{z_{n-N}^{-}}\left(
    d x_n \mid y_{n-N}, \ldots, y_n, u_{n-N},
    \ldots, u_{n-1}\right).
\end{align*}
For $n \geq N$, fixing $z_{n-N}^-$ to a constant  
$z^*\in P(\X)$, we obtain a 
new MDP.

Consider the following set $\hat{\mathcal{Z}}_{z^*}^N$ for a fixed $z^* \in \mathcal{P}(\mathbb{X})$
    \begin{align}\label{Z_set}
    \hat{\mathcal{Z}}_{z^*}^N=\left\{\left(z^*, y_{[0, N]}, u_{[0, N-1]}\right): y_{[0, N]} \in \mathbb{Y}^{N+1}, u_{[0, N-1]} \in \mathbb{U}^N\right\}
\end{align}
    such that the state at time $t$ is $\hat{z}_t^N=\left(z^*, I_t^N\right)$. Compared to the state $\hat{z}_t=\left(z_{t-N}^{-}, I_t^N\right)$, this approximate model uses $z^*$ as the predictor, no matter what the real predictor at time $t-N$ is.
    The approximate cost function is defined as
    $$
    \begin{aligned}
    \hat{c}\left(\hat{z}_t^N, u_t\right) & =\hat{c}\left(z^*, I_t^N, u_t\right)=\tilde{c}\left(\psi\left(z^*, I_t^N\right), u_t\right) \\
    & =\int_{\mathbb{X}} c\left(x_t, u_t\right) P^{z^*}\left(d x_t \mid y_t, \ldots, y_{t-N}, u_{t-1}, \ldots, u_{t-N}\right) .
    \end{aligned}
    $$
    
    We define the controlled transition model by
    $$
    \hat{\eta}^N\left(\hat{z}_{t+1}^N \mid \hat{z}_t^N, u_t\right)=\hat{\eta}^N\left(z^*, I_{t+1}^N \mid z^*, I_t^N, u_t\right):=\hat{\eta}\left(\mathcal{P}(\mathbb{X}), I_{t+1}^N \mid z^*, I_t^N, u_t\right) .
    $$
    
    For simplicity, if we assume $N=1$, then the transitions can be rewritten for some
    $$
    \begin{aligned}
    & I_{t+1}^N=\left(\hat{y}_{t+1}, \hat{y}_t, \hat{u}_t\right) \text { and } I_t^N=\left(y_t, y_{t-1}, u_{t-1}\right) \\
    & \hat{\eta}^N\left(z^*, \hat{y}_{t+1}, \hat{y}_t, \hat{u}_t \mid z^*, y_t, y_{t-1}, u_{t-1}, u_t\right)=\hat{\eta}\left(\mathcal{P}(\mathbb{X}), \hat{y}_{t+1}, \hat{y}_t, \hat{u}_t \mid z^*, y_t, y_{t-1}, u_{t-1}, u_t\right) \\
    & =\mathbf{1}_{\left\{y_t=\hat{y}_t, u_t=\hat{u}_t\right\}} P^{z^*}\left(\hat{y}_{t+1} \mid y_t, y_{t-1}, u_t, u_{t-1}\right) .
    \end{aligned}
    $$
    
    Denoting the optimal value function for the approximate model by $J_\beta^N$, we can write the following fixed point equation
    $$
    J_\beta^N\left(\hat{z}^N\right)=\min _{u \in \mathbf{U}}\left(\hat{c}\left(\hat{z}^N, u\right)+\beta \sum_{\hat{z}_1^N \in \hat{\mathcal{Z}}_{z^*}^N} J_\beta^N\left(\hat{z}_1^N\right) \hat{\eta}^N\left(\hat{z}_1^N \mid \hat{z}^N, u\right)\right) .
    $$
    
    By assuming $N=1$, we can rewrite the fixed point equation for some $\hat{z}_0^N=$ $\left(z^*, y_1, y_0, u_0\right)$ as
    $$
    J_\beta^N\left(z^*, y_1, y_0, u_0\right)=\min _{u_1 \in \mathbf{U}}\left(\hat{c}\left(z^*, y_1, y_0, u_0, u_1\right)+\beta \sum_{y_2 \in \mathbb{Y}} J_\beta^N\left(z^*, y_2, y_1, u_1\right) P^{z^*}\left(y_2 \mid y_1, y_0, u_1, u_0\right)\right) .
    $$
    
   Due to the finiteness of this approximate MDP, one can assume the existence of an optimal policy $\phi^N$ that satisfies this fixed point equation. Note that both $J_\beta^N$ and $\phi^N$ are defined on the finite set $\hat{\mathcal{Z}}_{z^*}^N$. However, we can simply extend them to the set $\hat{\mathcal{Z}}$ by defining
    \begin{align}
    & \tilde{J}_\beta^N(\hat{z})=\tilde{J}_\beta^N\left(z, y_1, y_0, u_0\right):=J_\beta^N\left(z^*, y_1, y_0, u_0\right) \\
    & \tilde{\phi}^N(\hat{z})=\tilde{\phi}^N\left(z, y_1, y_0, u_0\right):=\phi^N\left(z^*, y_1, y_0, u_0\right) \label{optimalfinite}
    \end{align} 
    for any $\hat{z}=\left(z, y_1, y_0, u_0\right) \in \hat{\mathcal{Z}}$.
    This also applies for $N > 1$. 
    The results in the following will be stated for arbitrary $N \in \mathbb{N}$.

\subsection{Convergence notions for probability measures.}
Let $\{\mu_n,\, n\in \mathbb{N}\}$ be a sequence in
$\mathcal{P}(\mathbb{X})$.
The sequence $\{\mu_n\}$ is said to  converge
to $\mu\in \mathcal{P}(\mathbb{X})$ \emph{weakly} if
\begin{align}\label{weakConvD}
 \int_{\mathbb{X}} f(x) \mu_n(dx)  \to \int_{\mathbb{X}}f(x) \mu(dx)
\end{align}
for every continuous and bounded $f: \mathbb{X} \to \mathbb{R}$.

For two probability measures $\mu,\nu \in
\mathcal{P}(\mathbb{X})$, the \emph{total variation} metric
is given by
\begin{align}
\|\mu-\nu\|_{TV}:= & 2 \sup_{B \in {\mathcal B}(\mathbb{X})}
|\mu(B)-\nu(B)| \nonumber \\
 =&  \sup_{f: \, \|f\|_{\infty} \leq 1} \bigg| \int f(x)\mu(dx) -
\int f(x)\nu(dx) \bigg|, \label{TValternative}
\end{align}
where the supremum is over all measurable real $f$ such that
$\|f\|_{\infty} = \sup_{x \in \mathbb{X}} |f(x)|\le 1$.

Finally, the bounded-Lipschitz metric 
$\rho_{BL}$ \cite[p.109]{villani2008optimal} 
can also be used to metrize weak convergence:
\begin{align}
\rho_{BL}(\mu,\nu) = \sup_{\|f\|_{BL}\leq1} \biggl| \int_{\mathbb{X}} f(x) \mu(dx) - \int_{\mathbb{X}} f(x) \nu(dx) \biggr|, \label{BLD}
\end{align}
where $\|f\|_{BL} := \|f\|_{\infty} + \norm{f}_{L}$, and  $\norm{f}_{L}= \sup_{x \neq x'} \frac{f(x) - f(x')}{d(x,x')}$ and $d$ is the metric on $\mathbb{X}$. 

When $\X$ is compact, one way to metrize $\mathcal{Z}$ under 
the weak convergence topology is via
the Kantorovich-Rubinstein metric 
(also known as the Wasserstein metric of order $1$) 
(\cite{Bog07}, Theorem 8.3.2) 
defined as follows 
\begin{align}\label{defkappanorm}
&W_1(\mu, \nu):=\sup \left\{\int_{\mathbb{X}} f(x) \mu(d x)-\int_{\mathbb{X}} f(x) \nu(d x) : f \in \operatorname{Lip}(\mathbb{X},1) \right\},
\end{align}
$\mu, \nu \in \mathcal{Z}$, where for $k \in \mathbb{N}$, $\operatorname{Lip}(\mathbb{X},k)=\{f:\mathbb{X}\to \mathbb{R},\; \norm{f}_{L}\leq k\}.$

\begin{definition}\label{Dobrushincoefficient}
    [\cite{dobrushin1956central}, Equation 1.16.] 
    For a stochastic kernel $K$ from a standard Borel space $S_1$ to another one $S_2$, so that $K: S_1 \rightarrow \mathcal{P}\left(S_2\right)$ is a Borel measurable function, we define the Dobrushin coefficient as:
$$
\delta(K)=\inf \sum_{i=1}^n \min \left(K\left(x, A_i\right), K\left(y, A_i\right)\right)
$$
where the infimum is over all $x, y \in S_1$ and all finite partitions $\left\{A_i\right\}_{i=1}^n$, $n \in \mathbb{N}$, of $S_2$.
\end{definition}
 
\section{Main Results}
Throughout this section, we will assume that $\mathbb{Y}$
 and $\mathbb{U}$ are finite.
\subsection{Difference in the value function in terms of a relaxed filter stability error}


In the following, we will have a relaxed counterpart under the Wasserstein metric.
\begin{assumption}\label{main_assumption}
    \noindent
    \begin{enumerate} [label=(\roman*)]
    \item \label{compactness}
    $(\X, d)$ is a compact metric space 
    with diameter $D$ (where $D=\sup_{x,y \in \mathbb{X}} d(x,y)$).
    \item \label{totalvar}
    The transition probability $\sT(\cdot \mid x, u)$ is 
    continuous in total variation in $(x, u)$, i.e., 
    for any $\left(x_n, u_n\right) \rightarrow(x, u), 
    \sT\left(\cdot \mid x_n, u_n\right) \rightarrow 
    \sT(\cdot \mid x, u)$ in total variation.
    \item \label{regularity}
    There exists 
    $\alpha \in R^{+}$ such that 
    $
    \left\|\mathcal{T}(\cdot \mid x, u)-\mathcal{T}\left(\cdot \mid x^{\prime}, u\right)\right\|_{T V} \leq \alpha d(x, x^{\prime})
    $
    for every $x,x' \in \mathbb{X}$, $u \in \mathbb{U}$.
    \item \label{CostLipschitz}
    There exists $K_1 \in \mathbb{R}^+$ such that
    $|c(x,u) - c(x',u)| \leq K_1 d(x,x')$
    for every $x,x' \in \mathbb{X}$, $u \in \mathbb{U}$.
    \item \label{bdd} The cost function $c$ is continuous and bounded.
     \end{enumerate}
\end{assumption}
    
\begin{theorem}\label{J_diff}
    Under Assumption \ref{main_assumption}-\ref{regularity}, Assumption \ref{main_assumption}-\ref{CostLipschitz} and Assumption \ref{main_assumption}-\ref{bdd},
    for $\hat{z}_0=\left(z_0^{-}, I_0^N\right)$, if a policy $\hat{\gamma}$ acts on the first $N$ step of the process that produces $I_0^N$, we then have
    $$
    E_{z_0^{-}}^{\hat{\gamma}}\left[\left|\tilde{J}_\beta^N\left(\hat{z}_0\right)-J_\beta^*\left(\hat{z}_0\right)\right| \mid I_0^N\right] \leq \left(K_1+\alpha \beta \frac{\|c\|_{\infty}}{1-\beta} \right) \sum_{t=0}^{\infty} \beta^t \bar{L}^N_t 
    $$
    where 
    \begin{align}\label{barLW1}
        \bar{L}_{t}^N:=\sup _{\hat{\gamma} \in \hat{\Gamma}} E_{z_0^{-}}^{\hat{\gamma}}\left[W_1\left(P^{z_t^{-}}\left(X_{t+N} \in \cdot \mid Y_{[t, t+N]}, U_{[t, t+N-1]}\right),P^{z^*}\left(X_{t+N} \in \cdot \mid Y_{[t, t+N]}, U_{[t, t+N-1]}\right)\right)\right].
    \end{align}
Here, $\hat{\Gamma}$ is defined as after equation (\ref{def_L_t}).
\end{theorem}

Toward the proof, we first state and prove a lemma.
\begin{lemma}\label{y_x}
Under Assumption \ref{main_assumption}-\ref{regularity}, 
for any $z, z^* \in \mathcal{P}(\mathbb{X})$ and for any 
\begin{align*}
    (y, u)_{[t, t-N]}:=\{y_t, \ldots, y_{t-N}, u_t, \ldots, u_{t-N} \} \in \mathbb{Y}^N \times \mathbb{U}^N,
\end{align*}
we have
\begin{align}
& \left\|P^z\left(Y_{t+1} \in \cdot \mid(y, u)_{[t, t-N]}\right)-P^{z^*}\left(Y_{t+1} \in \cdot \mid(y, u)_{[t, t-N]}\right)\right\|_{T V} \\
& \leq \alpha W_1 \left( P^z\left(X_t \in \cdot \mid y_{[t, t-N]}, u_{[t-1, t-N]}\right) , P^{z^*}\left(X_t \in \cdot \mid y_{[t, t-N]}, u_{[t-1, t-N]}\right) \right).
\end{align}
\end{lemma}
\begin{proof}
Let $f$ be a measurable function of $\mathbb{Y}$ such that $\|f\|_{\infty} \leq 1$. We write
\begin{align}
& \int f\left(y_{t+1}\right) P^z\left(d y_{t+1} \mid(y, u)_{[t, t-N]}\right)-\int f\left(y_{t+1}\right) P^{z^*}\left(d y_{t+1} \mid(y, u)_{[t, t-N]}\right) \\
& =\int f\left(y_{t+1}\right) Q \left(d y_{t+1} \mid x_{t+1}\right) \mathcal{T}\left(d x_{t+1} \mid x_t, u_t\right) P^z\left(d x_t \mid y_{[t, t-N]}, u_{[t-1, t-N]}\right) \\
& -\int f\left(y_{t+1}\right) Q \left(d y_{t+1} \mid x_{t+1}\right) \mathcal{T}\left(d x_{t+1} \mid x_t, u_t\right) P^{z^*}\left(d x_t \mid y_{[t, t-N]}, u_{[t-1, t-N]}\right) \\
& =\int h(x_t) P^z\left(d x_t \mid y_{[t, t-N]}, u_{[t-1, t-N]}\right) 
 -\int h(x_t) P^{z^*}\left(d x_t \mid y_{[t, t-N]}, u_{[t-1, t-N]}\right),
\end{align}
where $h(x_t)=\int f\left(y_{t+1}\right) Q \left(d y_{t+1} \mid x_{t+1}\right) \mathcal{T}\left(d x_{t+1} \mid x_t, u_t\right)$.
Since $h/\alpha \in \Lip (\X)$ because of  
the fact that $\int f\left(y_{t+1}\right) Q\left(d y_{t+1} \mid x_{t+1}\right)$ is bounded by one as a function of $x_{t+1}$ for $\|f\|_{\infty} \leq 1$
and $\|\mathcal{T}(\cdot \mid x, u)-\mathcal{T}(\cdot \mid x^{\prime}, u)\|_{T V} \leq \alpha d(x, x^{\prime})$
by Assumption \ref{main_assumption}-\ref{regularity}, we have 
\begin{align}
    & \int f\left(y_{t+1}\right) P^z\left(d y_{t+1} \mid(y, u)_{[t, t-N]}\right)-\int f\left(y_{t+1}\right) P^{z^*}\left(d y_{t+1} \mid(y, u)_{[t, t-N]}\right) \\
    & =\int h(x_t) P^z\left(d x_t \mid y_{[t, t-N]}, u_{[t-1, t-N]}\right) 
     -\int h(x_t) P^{z^*}\left(d x_t \mid y_{[t, t-N]}, u_{[t-1, t-N]}\right)\\
    & \leq \alpha W_1 \left(P^z\left(X_t \in \cdot \mid y_{[t, t-N]}, u_{[t-1, t-N]}\right) , P^{z^*}\left(X_t \in \cdot \mid y_{[t, t-N]}, u_{[t-1, t-N]}\right)\right)
\end{align}

Taking the supremum over all $\|f\|_{\infty} \leq 1$ concludes the proof.

\end{proof}

    \begin{proof}[Proof of Theorem \ref{J_diff}]
    For simplicity in presentation let $N=1$, the prof for the general $N \in \mathbb{N}$ applies identically.
    Let $\hat{z}_0=\left(z_0^{-}, y_1, y_0, u_0\right)$. Then we write
    \begin{align*}
    & \tilde{J}_\beta^N\left(\hat{z}_0\right)=J_\beta^N\left(z^*, y_1, y_0, u_0\right)\\
    & =\min _{u_1 \in \mathbf{U}}\left(\hat{c}\left(z^*, y_1, y_0, u_0, u_1\right)+\beta \sum_{y_2 \in \mathbb{Y}} J_\beta^N\left(z^*, y_2, y_1, u_1\right) P^{z^*}\left(y_2 \mid y_1, y_0, u_1, u_0\right)\right) \text {. }
    \end{align*}
    
    Furthermore,
    \begin{align*}
    & J_\beta^*\left(\hat{z}_0\right) =J_\beta^*\left(z_0^{-}, y_1, y_0, u_0\right) \\
    & =\min _{u_1 \in \mathbb{U}}\left(\hat{c}\left(z_0^{-}, y_1, y_0, u_0, u_1\right)+\beta \sum_{y_2 \in \mathbb{Y}} J_\beta^*\left(z_1^{-}\left(z_0^{-}, y_0, u_0\right), y_2, y_1, u_1\right) P^{z_0^{-}}\left(y_2 \mid y_1, y_0, u_1, u_0\right)\right) .
    \end{align*}
    
    Note that, for any $z \in \mathcal{P}(\mathbb{X})$, we have
    $$
    \tilde{J}_\beta^N\left(z, y_2, y_1, u_1\right)=\tilde{J}_\beta^N\left(z^*, y_2, y_1, u_1\right)=J_\beta^N\left(z^*, y_2, y_1, u_1\right) .
    $$
    
    In particular, we have that
    $$
    J_\beta^N\left(z^*, y_2, y_1, u_1\right)=\tilde{J}_\beta^N\left(z_1^{-}\left(z_0^{-}, y_0, u_0\right), y_2, y_1, u_1\right) .
    $$
    
    Hence, we can write the following
    \begin{align*}
    & \left|\tilde{J}_\beta^N\left(\hat{z}_0\right)-J_\beta^*\left(\hat{z}_0\right)\right| \leq \max _{u_1 \in \mathbb{U}}\left|\hat{c}\left(z^*, y_1, y_0, u_0, u_1\right)-\hat{c}\left(z_0^{-}, y_1, y_0, u_0, u_1\right)\right| \\
    & +\max _{u_1 \in \mathbb{U}} \beta\left|\sum_{y_2 \in \mathbb{Y}} J_\beta^N\left(z^*, y_2, y_1, u_1\right) P^{z^*}\left(y_2 \mid y_1, y_0, u_1, u_0\right)-\sum_{y_2 \in \mathbb{Y}} J_\beta^N\left(z^*, y_2, y_1, u_1\right) P^{z_0^{-}}\left(y_2 \mid y_1, y_0, u_1, u_0\right)\right| \\
    & +\max _{u_1 \in \mathbb{U}} \beta \sum_{y_2 \in \mathbb{Y}}\left|\tilde{J}_\beta^N\left(z_1^{-}\left(z_0^{-}, y_0, u_0\right), y_2, y_1, u_1\right)-J_\beta^*\left(z_1^{-}\left(z_0^{-}, y_0, u_0\right), y_2, y_1, u_1\right)\right| P^{z_0^{-}}\left(y_2 \mid y_1, y_0, u_1, u_0\right) .
    \end{align*}
    
    Note that, by the definition of $\hat{c}$, under Assumption \ref{main_assumption}-\ref{CostLipschitz}, we have
    \begin{align}
    & \left|\hat{c}\left(z^*, y_1, y_0, u_0, u_1\right)-\hat{c}\left(z_0^{-}, y_1, y_0, u_0, u_1\right)\right| \\
    & \leq K_1 W_1 \left(P^{z^*}\left(X_1 \in \cdot \mid y_1, y_0, u_0\right) , P^{z_0^{-}}\left(X_1 \in \cdot \mid y_1, y_0, u_0\right)\right)
    \end{align}
    
    If we denote $\hat{z}_1=(z_1^{-}(z_0^{-}, y_0, u_0), y_2, y_1, u_1)$, using Lemma \ref{y_x}, we can write:
    \begin{align}
    & E_{z_0^{-}}^\gamma\left[\left|\tilde{J}_\beta^N\left(\hat{z}_0\right)-J_\beta^*\left(\hat{z}_0\right)\right|\right] 
    \leq K_1 E_{z_0^{-}}^\gamma\left[W_1 \left(P^{z^*}\left(X_1 \in \cdot \mid Y_1, Y_0, U_0\right) , P^{z_0^{-}}\left(X_1 \in \cdot \mid Y_1, Y_0, U_0\right)\right)\right] \\
    & +\max _{u_1 \in \mathbb{U}} \beta\left\|J_\beta^N\right\|_{\infty} E_{z_0^{-}}^\gamma\left[\left\|P^{z_0^{-}}\left(y_2 \mid Y_1, Y_0, U_1, U_0\right)-P^{z^*}\left(y_2 \mid Y_1, Y_0, U_1, U_0\right)\right\|_{T V}\right] \\
    & +\max _{u_1 \in \mathbb{U}} \beta E_{z_0^{-}}^\gamma\left[\sum_{y_2 \in \mathbb{Y}}\left|\tilde{J}_\beta^N\left(\hat{z}_1\right)-J_\beta^*\left(\hat{z}_1\right)\right| P^{z_0^{-}}\left(y_2 \mid Y_1, Y_0, U_1, U_0\right)\right] \\
    &\leq\left(K_1+\beta\left\|J_\beta^N\right\|_{\infty} \alpha \right) \bar{L}_0+\max _{u_1 \in \mathbb{U}} \beta E_{z_0^{-}}^\gamma\left[\sum_{y_2 \in \mathbb{Y}}\left|\tilde{J}_\beta^N\left(\hat{z}_1\right)-J_\beta^*\left(\hat{z}_1\right)\right| P^{z_0^{-}}\left(y_2 \mid Y_1, Y_0, u_1, U_0\right)\right] \\
    & \leq\left(K_1+\alpha \beta\left\|J_\beta^N\right\|_{\infty} \right) \bar{L}_0+\sup _{\hat{\gamma} \in \hat{\Gamma}} \beta E_{z_0^{-}}^{\hat{\gamma}}\left[\left|\tilde{J}_\beta^N\left(\hat{z}_1\right)-J_\beta^*\left(\hat{z}_1\right)\right|\right]
    \end{align}
    The second to last line follows from Lemma \ref{y_x}.
    Then, following the same steps for $E_{z_0^{-}}^{\hat{\gamma}}\left[\left|\tilde{J}_\beta^N\left(\hat{z}_1\right)-J_\beta^*\left(\hat{z}_1\right)\right|\right]$ and repeating the procedure, one can see that
    $$
    E_{z_0^{-}}^\gamma\left[\left|\tilde{J}_\beta^N\left(\hat{z}_0\right)-J_\beta^*\left(\hat{z}_0\right)\right|\right] \leq\left(K_1+\alpha \beta\left\|J_\beta^N\right\|_{\infty} \right) \sum_{t=0}^{\infty} \beta^t \bar{L}_t
    $$, the result follows by noting that $\left\|J_\beta^N\right\|_{\infty} \leq \frac{\|c\|_{\infty}}{1-\beta}$. 
    \end{proof}

\subsection{Performance Loss of the Sliding Window Policy in the True System}
In this subsection, we will examine the error term that arises when finite window optimal policies are applied to the true system.

\begin{theorem}\label{J_diff_policy}
Under Assumption \ref{main_assumption}-\ref{regularity}, Assumption \ref{main_assumption}-\ref{CostLipschitz} and Assumption \ref{main_assumption}-\ref{bdd},
    for $\hat{z}_0=\left(z_0^{-}, I_0^N\right)$, if a policy $\hat{\gamma}$ acts on the first $N$ step of the process that produces $I_0^N$, we then have
    \begin{align}
    E_{z_0^{-}}^{\hat{\gamma}}\left[\left| J_\beta \left(\hat{z}_0, \tilde{\phi}^N\right)
    -J_\beta^*\left(\hat{z}_0\right)\right| \mid I_0^N\right]  \leq 2 \left(K_1+\alpha \beta \frac{\|c\|_{\infty}}{1-\beta} \right) \sum_{t=0}^{\infty} \beta^t \bar{L}^N_t. 
    \end{align}
\end{theorem}
\begin{proof}
For simplicity in presentation let $N=1$, the prof for the general $N \in \mathbb{N}$ applies identically.
We let $\hat{z}_0=\left(z_0^{-}, y_1, y_0, u_0\right)$. We denote the minimum selector for the approximate MDP by
$
u_1^N:=\tilde{\phi}^N\left(z_0^{-}, y_1, y_0, u_0\right)=\phi^N\left(z^*, y_1, y_0, u_0\right)
$
and write
\begin{align*}
& J_\beta\left(\hat{z}_0, \tilde{\phi}^N\right)  =J_\beta\left(z_0^{-}, y_1, y_0, u_0, \tilde{\phi}^N\right) \\
& =\hat{c}\left(z_0^{-}, y_1, y_0, u_0, u_1^N\right)+\beta \sum_{y_2 \in \mathbb{Y}} J_\beta\left(z_1^{-}\left(z_0^{-}, y_0, u_0\right), y_2, y_1, u_1^N, \tilde{\phi}^N\right) P^{z_0^{-}}\left(y_2 \mid y_1, y_0, u_1^N, u_0\right).
\end{align*}
Furthermore, we write the optimality equation for $\tilde{J}_\beta^N$ as follows
$$
\tilde{J}_\beta^N\left(\hat{z}_0\right)=\hat{c}\left(z^*, y_1, y_0, u_0, u_1^N\right)+\beta \sum_{y_2 \in \mathbb{Y}} \tilde{J}_\beta^N\left(z_1^{-}\left(z_0^{-}, y_0, u_0\right), y_2, y_1, u_1^N\right) P^{z^*}\left(y_2 \mid y_1, y_0, u_1^N, u_0\right).
$$
    With $\hat{z}_1:=\left(z_1^{-}\left(z_0^{-}, y_0, u_0\right), y_2, y_1, u_1^N\right)$, let us write the following
    \begin{align*}
    &\left|J_\beta\left(\hat{z}_0, \tilde{\phi}^N\right)-\tilde{J}_\beta^N\left(\hat{z}_0\right)\right| \leq \left|\hat{c}\left(z_0^{-}, y_1, y_0, u_0, u_1^N\right)-\hat{c}\left(z^*, y_1, y_0, u_0, u_1^N\right)\right| \\
    & + \beta\left|\sum_{y_2 \in \mathbb{Y}} \tilde{J}_\beta^N\left(\hat{z}_1 \right) P^{z_0^{-}}\left(y_2 \mid y_1, y_0, u_1^N, u_0\right)- \sum_{y_2 \in \mathbb{Y}} \tilde{J}_\beta^N\left(\hat{z}_1 \right) P^{z^*}\left(y_2 \mid y_1, y_0, u_1^N, u_0\right)\right| \\
    & + \beta \sum_{y_2 \in \mathbb{Y}}\left| J_\beta \left(\hat{z}_1, \tilde{\phi}^N \right) -\tilde{J}_\beta^N\left(\hat{z}_1 \right)\right| P^{z_0^{-}}\left(y_2 \mid y_1, y_0, u_1^N, u_0\right) .
    \end{align*}
    Note that, by the definition of $\hat{c}$, we have
    \begin{align*}
    & \left|\hat{c}\left(z^*, y_1, y_0, u_0, u_1\right)-\hat{c}\left(z_0^{-}, y_1, y_0, u_0, u_1\right)\right|  \leq K_1 W_1 \left(P^{z^*}\left(X_1 \in \cdot \mid y_1, y_0, u_0\right) , P^{z_0^{-}}\left(X_1 \in \cdot \mid y_1, y_0, u_0\right)\right).
    \end{align*}
Using Lemma \ref{y_x}, we can express this as follows
$$
\begin{aligned}
& E_{z_0^{-}}^{\hat{\gamma}}\left[\left|J_\beta\left(\hat{z}_0, \tilde{\phi}^N\right)-\tilde{J}_\beta^N\left(\hat{z}_0\right)\right|\right] \leq \sup _{\hat{\gamma} \in \hat{\Gamma}} E_{z_0^{-}}^{\hat{\gamma}}\left[\left|\hat{c}\left(z_0^{-}, Y_1, Y_0, U_0, U_1\right)-\hat{c}\left(z^*, Y_1, Y_0, U_0, U_1\right)\right|\right] \\
& +\sup _{\hat{\gamma} \in \hat{\Gamma}} E_{z_0^{-}}^{\hat{\gamma}}\left[\beta \left| \sum_{y_2 \in \mathbb{Y}} \tilde{J}_\beta^N \left(\hat{z}_1\right) P^{z_0^{-}}\left(y_2 \mid Y_1, Y_0, U_1, U_0\right)-\sum_{y_2 \in \mathbb{Y}} \tilde{J}_\beta^N\left(\hat{z}_1\right) P^{z^*}\left(y_2 \mid Y_1, Y_0, U_1, U_0\right)\right|\right] \\
&+ \sup _{\hat{\gamma} \in \hat{\Gamma}} E_{z_0^{-}}^{\hat{\gamma}}\left[\beta \sum_{y_2 \in \mathbb{Y}}  \left|  J_\beta (\hat{z}_1, \tilde{\phi}^N ) - \tilde{J}_\beta^N\left(\hat{z}_1\right)  \right|  P^{z_0^{-}}\left(y_2 \mid Y_1, Y_0, U_1, U_0\right) \right] \\
& \leq K_1 E_{z_0^{-}}^\gamma\left[W_1 \left(P^{z^*}\left(X_1 \in \cdot \mid Y_1, Y_0, U_0\right) , P^{z_0^{-}}\left(X_1 \in \cdot \mid Y_1, Y_0, U_0\right)\right)\right] \\
    & +\max _{u_1 \in \mathbb{U}} \beta\left\|\tilde{J}_\beta^N\right\|_{\infty} E_{z_0^{-}}^{\hat{\gamma}}\left[\left\|P^{z_0^{-}}\left(y_2 \mid Y_1, Y_0, U_1, U_0\right)-P^{z^*}\left(y_2 \mid Y_1, Y_0, U_1, U_0\right)\right\|_{T V}\right] \\
    & +\max _{u_1 \in \mathbb{U}} \beta E_{z_0^{-}}^{\hat{\gamma}} \left[\sum_{y_2 \in \mathbb{Y}} \left|  J_\beta (\hat{z}_1, \tilde{\phi}^N ) - \tilde{J}_\beta^N\left(\hat{z}_1\right)  \right| P^{z_0^{-}}\left(y_2 \mid Y_1, Y_0, U_1, U_0\right)\right] \\
    &\leq\left(K_1+\beta\left\|\tilde{J}_\beta^N\right\|_{\infty} \alpha \right) \bar{L}_0 +\beta \sup _{\hat{\gamma} \in \hat{\Gamma}} E_{z_0^{-}}^{\hat{\gamma}}\left[\left|J_\beta\left(\hat{z}_1, \tilde{\phi}^N\right)-\tilde{J}_\beta^N\left(\hat{z}_1\right)\right|\right]
\end{aligned}
$$
Following the same steps for $E_{z_0^{-}}^{\hat{\gamma}}\left[\left|J_\beta\left(\hat{z}_1, \tilde{\phi}^N\right)-\tilde{J}_\beta^N\left(\hat{z}_1\right)\right|\right]$ and repeating the same procedure with $\left\|\tilde{J}_\beta^N\right\|_{\infty} \leq \frac{\|c\|_{\infty}}{1-\beta}$, one can conclude that
$$
E_{z_0}^{\hat{\gamma}}\left[\left|J_\beta\left(\hat{z}_0, \tilde{\phi}^N\right)-\tilde{J}_\beta^N\left(\hat{z}_0\right)\right|\right] \leq \left(K_1+\alpha \beta \frac{\|c\|_{\infty}}{1-\beta} \right) \sum_{t=0}^{\infty} \beta^t \bar{L}^N_t
$$
Now, we return to the theorem statement to write the following:
$$
\begin{aligned}
E_{z_0^{-}}^{\hat{\gamma}}\left[\left|J_\beta\left(\hat{z}_0, \tilde{\phi}^N\right)-J_\beta^*\left(\hat{z}_0\right)\right|\right] & \leq E_{z_0^{-}}^{\hat{\gamma}}\left[\left|J_\beta\left(\hat{z}_0, \tilde{\phi}^N\right)-\tilde{J}_\beta^N(\hat{z}_0)\right|\right]+E_{z_0^{-}}^{\hat{\gamma}}\left[\left|\tilde{J}_\beta^N(\hat{z}_0)-J_\beta^*(\hat{z}_0)\right|\right] \\
& \leq 2  \left(K_1+\alpha \beta \frac{\|c\|_{\infty}}{1-\beta} \right) \sum_{t=0}^{\infty} \beta^t \bar{L}^N_t
\end{aligned}
$$
The last step follows from Theorem \ref{J_diff}.
\end{proof}
\section{Refined Upper Bounds for Filter Stability Terms}
In this section, we provide upper bounds and explicit conditions on $\bar{L}^N$ (\ref{barLW1}) and $\bar{L}_{TV}^N$ (\ref{TVUnifB}) to satisfy convergence to zero as $N \to \infty$.

\subsection{Upper bounds on filter stability term in Wasserstein distance in expectation} 
    Recall that \begin{align*}
        \bar{L}_{t}^N:=\sup _{\hat{\gamma} \in \hat{\Gamma}} E_{z_0^{-}}^{\hat{\gamma}}\left[W_1\left(P^{z_t^{-}}\left(X_{t+N} \in \cdot \mid Y_{[t, t+N]}, U_{[t, t+N-1]}\right),P^{z^*}\left(X_{t+N} \in \cdot \mid Y_{[t, t+N]}, U_{[t, t+N-1]}\right)\right)\right].
    \end{align*}
\begin{lemma}
    Under Assumption \ref{main_assumption},
    \begin{align*}
    \bar{L}_{t}^N \leq \alpha \frac{D}{2} (2-\delta(Q)) \bar{L}_{t}^{N-1}
\end{align*}
\end{lemma}
\begin{proof}
\begin{align*}
&\bar{L}_{t}^N =\sup _{\hat{\gamma} \in \hat{\Gamma}} E_{z_0^{-}}^{\hat{\gamma}}\left[W_1\left(P^{z_t^{-}}\left(X_{t+N} \in \cdot \mid Y_{[t, t+N]}, U_{[t, t+N-1]}\right),P^{z^*}\left(X_{t+N} \in \cdot \mid Y_{[t, t+N]}, U_{[t, t+N-1]}\right)\right)\right]\\    
&\leq \sup _{\hat{\gamma} \in \hat{\Gamma}, u\in \U} E_{z_0^{-}}^{\hat{\gamma}}
\left[W_1\left(P^{z_t^{-}}\left(X_{t+N} \in \cdot \mid Y_{[t, t+N]}, U_{[t, t+N-2]},U_{t+N-1}=u \right),\right.\right. \\
&\qquad\qquad\qquad\qquad \left.\left. P^{z^*}\left(X_{t+N} \in \cdot \mid Y_{[t, t+N]}, U_{[t, t+N-2]}, U_{t+N-1}=u\right)\right)\right]
.\end{align*}
For a fixed sequence of $y_{[t, t+N-1]}$ and $u_{[t, t+N-2]}$, we have
\begin{align*}
&E_{Y_{t+N}}\left[W_1\left(P^{z_t^{-}}\left(X_{t+N} \in \cdot \mid y_{[t, t+N-1]},Y_{t+N},  u_{[t, t+N-2]}, u\right),P^{z^*}\left(X_{t+N} \in \cdot \mid y_{[t, t+N-1]},Y_{t+N},  u_{[t, t+N-2]}, u\right)\right)\right] \\
&=\int_{\Y} P (Y_{t+N}\in dy|z_0^1, u) W_1\left(P^{z_0^1}\left(X_1 \in \cdot \mid y, u \right),P^{z_0^2}\left(X_1 \in \cdot \mid y, u\right)\right),
\end{align*}
where 
\begin{align*}
    z_0^1= P^{z_t^{-}}\left(X_{t+N-1} \in \cdot \mid y_{[t, t+N-1]}, u_{[t, t+N-2]} \right)
\end{align*}
and 
\begin{align*}
    z_0^2= P^{z^*}\left(X_{t+N-1} \in \cdot \mid y_{[t, t+N-1]}, u_{[t, t+N-2]} \right).
\end{align*}
Therefore, 
    \begin{align*}
    &\int_{\Y} P (Y_{t+N}\in dy|z_0^1, u) W_1\left(P^{z_0^1}\left(X_1 \in \cdot \mid y, u \right),P^{z_0^2}\left(X_1 \in \cdot \mid y, u\right)\right)\\
    &=\int_{\Y} P (Y_{t+N}\in dy|z_0^1, u) W_1\left(F(z_0^1, u, y), F(z_0^2, u, y)\right)\\
    &=\int_{\mathbb{Y}} \sup_{g \in \operatorname{Lip}(\mathbb{X})}\left(\int_{\mathbb{X}} g(x_1)w_{y}(dx_1)\right)P\left(d y \mid  z_0^1, u\right),
   \end{align*}
    where $F(z_0^1, u, y)(\cdot)=P^{z_0^1}(X_1 \in \cdot \mid y, u )$,
    $F(z_0^2, u, y)(\cdot)=P^{z_0^2}(X_1 \in \cdot \mid y, u)$
    and
    $w_{y}=F(z_0^1, u, y)- F(z_0^2, u, y)$ 
    which is a signed measure on $\mathbb{X}$.
$\Lip(\mathbb{X})$ is closed, uniformly bounded and equicontinuos with respect to the sup-norm topology, 
so by the Arzela-Ascoli theorem $\Lip(\mathbb{X})$ is compact. 
Since a continuous function on a compact set attains its supremum,
 the set
$$
A_y:=\left\{ \arg\sup_{g \in \operatorname{Lip}(\mathbb{X})}\left(\int_{\mathbb{X}} g(x)w_{y}(dx)\right)\right\}
$$
is nonempty for every $y\in \mathbb{Y}$. The integral is continuous under respect to sup-norm, i.e., 
$$ 
\left|\int_{\mathbb{X}} g(x)w_{y}(dx)-\int_{\mathbb{X}} h(x)w_{y}(dx)\right|\leq\norm{g-h}_\infty \quad \forall g,h\in \Lip(\mathbb{X}). 
$$
Then, $A_y$ is a closed set under sup-norm.
$\mathbb{Y}$ and $\operatorname{Lip}(\mathbb{X})$ are Polish spaces, and define $\Gamma=(y,A_y)$. 
$A_y$ is closed for each $y\in \mathbb{Y}$ and $\Gamma$ is Borel measurable. 
So by the Measurable Selection Theorem \footnote{[\cite{himmelberg1976optimal}, Theorem 2][Kuratowski Ryll-Nardzewski Measurable Selection Theorem]
    Let $\mathbb{X}, \mathbb{Y}$ be Polish spaces and $\Gamma=(x, \psi(x))$ where $\psi(x) \subset \mathbb{Y}$ be such that, $\psi(x)$ is closed for each $x \in \mathbb{X}$ and let $\Gamma$ be a Borel measurable set in $\mathbb{X} \times \mathbb{Y}$. Then, there exists at least one measurable function $f: \mathbb{X} \rightarrow \mathbb{Y}$ such that $\{(x, f(x)), x \in \mathbb{X}\} \subset \Gamma$.
    }, 
    choose a measurable 
    $g_y\in \arg\sup_{g \in \operatorname{Lip}(\mathbb{X})}\left(\int_{\mathbb{X}} g(x)w_{y}(dx)\right).$ 
    
    After that,
    \begin{align*}
    &\int_{\mathbb{Y}} \sup_{g \in \operatorname{Lip}(\mathbb{X})}\left(\int_{\mathbb{X}} g(x_1)w_{y}(dx_1)\right)P\left(d y \mid z_0^1, u\right)\nonumber\\
    &=\int_{\mathbb{Y}} \int_{\mathbb{X}} g_{y}(x_1)F(z_0^1, u, y) (dx_1) P(d y \mid z_0^1, u)\\
    &-\int_{\mathbb{Y}}\int_{\mathbb{X}} g_{y}(x_1)F(z_0^2, u, y)(dx_1) P(d y \mid z_0^2, u)\\
    &+\int_{\mathbb{Y}}\int_{\mathbb{X}} g_{y}(x_1)F(z_0^2, u, y)(dx_1) P(d y \mid z_0^2, u)\\
    &-\int_{\mathbb{Y}}\int_{\mathbb{X}} g_{y}(x_1)F(z_0^2, u, y)(dx_1) P (d y \mid z_0^1, u )
    \end{align*}
    
    For the first term, we can write by smoothing
    \begin{align}
    &\int_{\mathbb{Y}} \int_{\mathbb{X}} g_{y}(x_1)F(z_0^1, u, y) (dx_1) P(d y \mid z_0^1, u)\\
    &-\int_{\mathbb{Y}}\int_{\mathbb{X}} g_{y}(x_1)F(z_0^2, u, y)(dx_1) P(d y \mid z_0^2, u)\\
    &=\int_{\mathbb{X}} \int_{\mathbb{X}} \omega(x_1) \mathcal{T}(d x_1 \mid x_0, u) z_0^1(dx_0)
    -\int_{\mathbb{X}}\int_{\mathbb{X}}\omega(x_1) \mathcal{T}(d x_1 \mid x_0, u) z_0^2(dx_0) \label{sup-1}
    \end{align}
    where 
    \begin{align*}
    \omega(x_1)=\int_{\mathbb{Y}} g_{y}(x_1)Q\left(d y \mid x_1\right).
    \end{align*}
    
    Since, for any $g_y:\mathbb{X}\to R$ such that $\norm{g_y}_L\leq 1$, we know that $|g_y(x)-g_y(x^{\prime})|\leq d(x,x^{\prime})\leq \operatorname{diam}(\mathbb{X})= D$ 
    and subtracting a constant from $g_y$ does not change the expression (\ref{sup-1}), without any loss of generality, we can assume that $\norm{g_y}_\infty\leq D/2$, 
    so that $\norm{\omega}_\infty\leq D/2$.

    For any $x^\prime, x^{\prime\prime} \in \mathbb{X}$, 
    \begin{align*}
    &\int_{\mathbb{X}} \omega(x) \mathcal{T}\left(d x \mid x^{\prime\prime}, u\right)-\int_{\mathbb{X}}\omega(x) \mathcal{T}\left(d x \mid x^{\prime}, u\right)\\
    &\leq \norm{\omega}_\infty \left\|\mathcal{T}(\cdot \mid x^{\prime\prime}, u)-\mathcal{T}\left(\cdot \mid x^{\prime}, u\right)\right\|_{T V}\\
    &\leq \norm{\omega}_\infty\alpha d(x^{\prime\prime},x^{\prime})\\
    &\leq\alpha \frac{D}{2} d(x^{\prime\prime}, x^{\prime})
    \end{align*}
    
    So, by definition of the $W_1$, we have
    \begin{align}\label{sup-2-1}
    &\int_{\mathbb{X}} \int_{\mathbb{X}} \omega(x_1) \mathcal{T}(d x_1 \mid x_0, u) z_0^1(dx_0)
    -\int_{\mathbb{X}}\int_{\mathbb{X}}\omega(x_1) \mathcal{T}(d x_1 \mid x_0, u) z_0^2(dx_0) \nonumber\\
    &\leq  \alpha \frac{D}{2} W_{1}(z_0^1, z_0^2).
    \end{align}

    For the second term,  we can write 
    \begin{align}\label{sup-2}
    &\int_{\mathbb{Y}}\int_{\mathbb{X}} g_{y}(x_1)F(z_0^2, u, y)(dx_1) P(d y \mid z_0^2, u)\nonumber\\
    &-\int_{\mathbb{Y}}\int_{\mathbb{X}} g_{y}(x_1)F(z_0^2, u, y)(dx_1) P(d y \mid z_0^1, u) \nonumber\\
    &\leq \frac{D}{2} \left\|P(\cdot \mid z_0^1, u)-P(\cdot \mid z_0^2, u)\right\|_{T V}
    \end{align}
\begin{align}\label{TV1}
    &\left\|P(\cdot \mid z_0^1, u)-P(\cdot \mid z_0^2, u)\right\|_{T V}
    =\sup _{\|g\|_{\infty} \leq 1}\left|\int g(y_1) P(d y_1 \mid z_0^1, u )-\int g\left(y_1\right) P(d y_1 \mid z_0^2, u )\right|\nonumber \\
    &=\sup _{\|g\|_{\infty} \leq 1}\left|\int g\left(y_1\right) Q\left(d y_1 \mid x_1\right) \mathcal{T}(d x_1 \mid z_0^1, u )
    -\int g\left(y_1\right) Q\left(d y_1 \mid x_1\right) \mathcal{T}(d x_1 \mid z_0^2, u )\right|\nonumber\\
    &\leq (1-\delta(Q))\left\|\mathcal{T}(d x_1 \mid z_0^1, u)-\mathcal{T}(d x_1 \mid z_0^2, u )\right\|_{T V}
\end{align}
    by Dobrushin contraction theorem \cite{dobrushin1956central}, where $\mathcal{T}(d x_1 \mid z_0, u) =\int \T (d x_1 \mid x_0, u) z_0(dx_0)$.
 
    \begin{align}\label{TV2}
        &\left\|\mathcal{T}(d x_1 \mid z_0^1, u)-\mathcal{T}(d x_1 \mid z_0^2, u)\right\|_{T V}\nonumber\\ 
        &= \sup _{\|g\|_{\infty} \leq 1}\left(\int g\left(x_1\right) \mathcal{T}(d x_1 \mid z_0^1, u)-\int g\left(x_1\right)\mathcal{T}(d x_1 \mid z_0^2, u)\right)\nonumber\\
        &=\sup _{\|g\|_{\infty} \leq 1}\left(\int\left(\int g\left(x_1\right) \T(d x_1 \mid x_0, u)z_0^1(dx_0)-\int g\left(x_1\right) \T(d x_1 \mid x_0, u)z_0^2(dx_0)\right)\right)\nonumber\\
        &= \sup _{\|g\|_{\infty} \leq 1}\left(\int \tilde{g_g}(x_0)z_0^1(dx_0)- \tilde{g_g}(x_0)z_0^2(dx_0)\right)
        \end{align}
        where
        \begin{align*}
            \tilde{g_g}(x)= \int g\left(x_1\right) \T (d x_1 \mid x, u).
        \end{align*}
        $\tilde{g_g}/\alpha \in Lip(\mathbb{X})$, since $|\tilde{g_g}(x_0)-\tilde{g_g}(x_0^\prime)|\leq \|\mathcal{T}(d x_1 \mid x_0^{\prime}, u)-\mathcal{T}(d x_1 \mid x_0,u)\|_{T V}\leq \alpha d(x_0,x_0^{\prime})$.
        
        Then, by inequality (\ref{TV1}) and (\ref{TV2}) we can write 
        \begin{align}\label{TV}
             \frac{D}{2} \left\|P(\cdot \mid z_0^1, u)-P(\cdot \mid z_0^2, u)\right\|_{T V}
            \leq \alpha \frac{D}{2} (1-\delta(Q))W_{1}( z_0^1,  z_0^2)
        \end{align}

Thus,

\begin{align*}
    \bar{L}_{t}^N \leq \alpha \frac{D}{2} (2-\delta(Q)) \bar{L}_{t}^{N-1}
\end{align*}

\end{proof}

\begin{corollary}\label{Corl}
    Under Assumption \ref{main_assumption},
\begin{align*}
    \bar{L}_{t}^N \leq \frac{D}{2} \left(\frac{\alpha D (2-\delta(Q))}{2}  \right)^N
\end{align*}
\end{corollary}
\begin{corollary}[Corollary of Theorem \ref{J_diff}]
    Under Assumption \ref{main_assumption},
    for $\hat{z}_0=\left(z_0^{-}, I_0^N\right)$, if a policy $\hat{\gamma}$ acts on the first $N$ step of the process that produces $I_0^N$, we then have
    $$
    E_{z_0^{-}}^{\hat{\gamma}}\left[\left|\tilde{J}_\beta^N\left(\hat{z}_0\right)-J_\beta^*\left(\hat{z}_0\right)\right| \mid I_0^N\right] \leq \left(K_1 (1-\beta) +\alpha \beta \|c\|_{\infty}\right) \frac{D}{2} \left(\frac{\alpha D (2-\delta(Q))}{2}  \right)^N. 
    $$
\end{corollary}
\begin{corollary}[Corollary of Theorem \ref{J_diff_policy}]\label{Corollary_j_diff}
    Under Assumption \ref{main_assumption},
    for $\hat{z}_0=\left(z_0^{-}, I_0^N\right)$, if a policy $\hat{\gamma}$ acts on the first $N$ step of the process that produces $I_0^N$, we then have
    \begin{align}
        E_{z_0^{-}}^{\hat{\gamma}}\left[\left|\tilde{J}_\beta^N\left(\hat{z}_0, \tilde{\phi}^N\right)
        -J_\beta^*\left(\hat{z}_0\right)\right| \mid I_0^N\right]  \leq 2 \left(K_1 (1-\beta) +\alpha \beta \|c\|_{\infty}\right) \frac{D}{2} \left(\frac{\alpha D (2-\delta(Q))}{2}  \right)^N. 
    \end{align}
\end{corollary}

\subsection{Upper bounds on the sample path-wise filter stability term in total variation} 

In this section, we aim to provide a uniform upper bound on the total variation distance using the Hilbert projective metric \cite{le2004stability}, but with a controlled version. Consequently, this will also naturally provide an upper bound for the expectation of the total variation distance by dominated convergence.
Recall (\ref{TVUnifB}):
\begin{align}
    \bar{L}_{TV}^N:\sup _{z \in \mathcal{P}(\mathbb{X})} \sup _{y_{[0, N]}, u_{[0, N-1]}}
    \left\|  
        P^{z}\left(X_{N} \in \cdot \mid {y}_{[0, N]}, u_{[0, N-1]}\right)-
     P^{z^*}\left(X_{N} \in \cdot \mid {y}_{[0, N]}, u_{[0, N-1]}\right)  \right\|_{T V} .
\end{align}
For $N\geq 1$ define probability measures for given initial measures $z$ and $z^*$:
\begin{align}Z_{N}=P^{z}\left(X_{N} \in \cdot \mid {y}_{[0, N]}, u_{[0, N-1]}\right)\end{align}
and \begin{align}Z^*_{N}=P^{z^*}\left(X_{N} \in \cdot \mid {y}_{[0, N]}, u_{[0, N-1]}\right).\end{align}


\begin{definition}
    Two non-negative measures $\mu, \nu$ on $(\mathbb{X},\B(\mathbb{X}))$ are comparable, if there exist positive constants $0<a \leq b$, such that
    $$
    a \nu(A) \leq \mu(A) \leq b \nu(A)
    $$
    for any Borel subset $A \subset \mathbb{X}$.
\end{definition}

\begin{definition}[Mixing kernel]
    The non-negative kernel $K$ defined on $\mathbb{X}$ is mixing, if there exists a constant $0<\varepsilon \leq 1$, and a non-negative measure $\lambda$ on $\mathbb{X}$, such that
    $$
    \varepsilon \lambda(A) \leq K(x, A) \leq \frac{1}{\varepsilon} \lambda(A)
    $$
    for any $x \in \mathbb{X}$, and any Borel subset $A \subset \mathbb{X}$.
    \end{definition}

\begin{definition}(Hilbert metric).  Let $\mu, \nu$ be two non-negative finite measures. We define the Hilbert metric on such measures as
    \begin{equation}
    h(\mu, \nu)= \begin{cases}\log \left(\frac{\sup _{A \mid \nu(A)>0} \frac{\mu(A)}{\nu(A)}}{\inf _{A \mid \nu(A)>0} \frac{\mu(A)}{\nu(A)}}\right) & \text { if } \mu, \nu \text { are comparable } \\ 0 & \text { if } \mu=\nu=0 \\ \infty & \text { else }\end{cases}
    \end{equation}
\end{definition}

Note that $h(a\mu, b\nu) = h(\mu, \nu)$ for any positive scalars $a, b$. Therefore, the Hilbert metric is a useful metric for nonlinear filters since it is invariant under normalization, and the following lemma demonstrates that it bounds the total-variation distance.

\begin{lemma}[\cite{le2004stability}, Lemma 3.4.]\label{h-TV}
    Let $\mu, \nu$ be two non-negative finite measures,
    \begin{enumerate}
    \item[i.] $
    \left\|\mu-\nu\right\|_{TV} \leq \frac{2}{\log 3} h\left(\mu, \nu\right) .
    $
    \item[ii.] 
    If the nonnegative kernel $K$ is a mixing kernel with constant $\epsilon$, then 
    $
    h\left(K \mu, K \nu\right) \leq \frac{1}{\varepsilon^2}\left\|\mu-\nu\right\|_{TV}.
    $
    \end{enumerate}
\end{lemma}

\begin{lemma}[\cite{le2004stability}, Lemma 3.8]\label{Birkoff} 
    (Birkhoff contraction coefficient). 
    The nonnegative linear operator $\tau$ on $\mathcal{M}^{+}(\mathbb{X})$ (positive measures on $\mathbb{X}$) 
    associated with a nonnegative kernel $K$ defined on $\mathbb{X}$
    $$
    \tau(K):=\sup _{0<h\left(\mu, \nu\right)<\infty} \frac{h\left(K \mu, K \nu\right)}{h\left(\mu, \nu\right)}=\tanh \left[\frac{1}{4} H(K)\right]
    $$
    where
    $$
    H(K):=\sup _{\mu, \nu} h\left(K \mu, K \nu\right)
    $$
    is over nonnegative measures, is a contraction (called the Birkhoff contraction coefficient) , under the Hilbert metric if $H(K)<\infty$ (which implies $\tau(K)<1$).
\end{lemma}

Recall that $
F(z, y, u)(\cdot)=\operatorname{Pr}\left\{X_{n+1} \in \cdot \mid Z_n=z, Y_{n+1}=y, U_n=u\right\}
$.
\begin{assumption} \label{mixing_kernel_con}
    \begin{enumerate} [label=(\roman*)]
    \item $Q(y|x)\geq \epsilon$ for every $x\in \X$ and $y\in \Y$.
    \item The transition kernel $\T(.|.,u)$ is a mixing kernel
    for every $u\in\U$.
\end{enumerate} 
\end{assumption}
Building on \cite{le2004stability}, it can be shown that the following result holds also for the controlled case using the contraction property.
\begin{lemma}\label{clm}
    Under Assumption \ref{mixing_kernel_con}, 
    there exists a constant $r < 1$ such that
    \begin{align}
        h(F(\mu, y,u), F(\nu, y,u))\leq r h(\mu, \nu)
    \end{align}
    for every comparable $\mu,\nu\in \Z$, and for every
    $u\in \U$ and $y\in \Y$. Here $r=\frac{1-\epsilon_{u}^2\epsilon }{1+\epsilon_{u}^2\epsilon},$ $\epsilon_{u}$ is the mixing constant of the kernel $\T(.|.,u)$.
\end{lemma}

Using this lemma, we can bound the difference between 
$Z_N$ and $Z^*_N$
for different initial measures 
$z$ and $z^*$. 
\begin{lemma}\label{h_bound}
    Under Assumption \ref{mixing_kernel_con},
    there exists a constant $r<1$  such that 
    \begin{align}
        h(Z_N, Z^*_N)\leq r^{N-1} h(Z_1, Z^*_1).
    \end{align} 
Here $r=\sup_{u\in \U} \frac{1-\epsilon_{u}^2\epsilon }{1+\epsilon_{u}^2\epsilon},$ $\epsilon_{u}$ is the mixing constant of the kernel $\T(.|.,u)$.
\end{lemma}
\begin{proof}
    Under Assumption \ref{mixing_kernel_con}, 
    $Z_1$ and $Z^*_1$ are comparable and the proof follows from Lemma \ref{clm}.
\end{proof}
This implies that $\bar{L}_{TV}^N$ converges to $0$ geometrically fast.
\begin{theorem}\label{d}
    Under Assumption \ref{mixing_kernel_con}, 
    there exists a constant $r<1$ and $K$ such that 
    \begin{align}
        \bar{L}_{TV}^N \leq r^{N-1} K.
    \end{align}
    Here, $K=\frac{2}{\log 3} \sup h(Z_1, Z_1^*)$ and $r=\sup_{u\in \U} \frac{1-\epsilon_{u}^2\epsilon }{1+\epsilon_{u}^2\epsilon}$.
\end{theorem}
\begin{proof}
    By Lemma \ref{h_bound}, we have $ h(Z_N, Z^*_N) \leq r^{N-1} \sup h(Z_1, Z_1^*)$. 
The proof follows from Lemma \ref{h-TV}.
\end{proof}

The same bound can be given for $L^N_t$. 
Since $L^N_t$ involves taking the expectation and each term is bounded by the above result.
    \begin{lemma}\label{L_t_n}
        Under Assumption \ref{mixing_kernel_con}, 
        there exists a constant $r<1$ and $K$ such that 
        \begin{align}
            L^N_t\leq  r^{N-1} K.
        \end{align}
    \end{lemma}
Therefore, we conclude that under Assumption \ref{mixing_kernel_con}, as $N$ goes to infinity, the finite window policy becomes near-optimal. From the inequality (\ref{Kara_yuksel}) and Lemma \ref{L_t_n}, we can establish that the convergence rate is geometric. This complements the bounds given in \cite[Theorem 3]{kara2021convergence}:

\begin{corollary}\label{cor_kara_y}
 Under Assumption \ref{mixing_kernel_con}, 
    there exists a constant $r<1$ and $K$ such that 
    \begin{align}
    E_{z_0^{-}}^{\hat{\gamma}}\left[\left|\tilde{J}_\beta^N\left(\hat{z}_0, \tilde{\phi}^N\right)
    -J_\beta^*\left(\hat{z}_0\right)\right| \mid I_0^N\right] 
    \leq \frac{2 \|c\|_{\infty}}{(1-\beta)^2} r^{N-1} K.
    \end{align}
        Here, $K=\frac{2}{\log 3} \sup h(Z_1, Z_1^*)$ and $r=\sup_{u\in \U} \frac{1-\epsilon_{u}^2\epsilon }{1+\epsilon_{u}^2\epsilon}$.
\end{corollary}

Theorem \ref{d} provides an upper bound for the approximation error discussed in \cite{kara2020near}. Unlike the method used here, \cite{kara2020near} establishes the near-optimality of finite window policies by quantizing the belief space using the nearest map under the bounded-Lipschitz metric. In this approach, at each step, the nearest element in the set 
$\hat{\mathcal{Z}}_{z^*}^N$ (see (\ref{Z_set})) is selected based on the bounded-Lipschitz metric. As opposed to the method we use in this paper, direct discretization of the probability measure space requires one to track the belief state at each time step, which is not needed in our current paper. 

For the approximation policy used, the following result holds:
   \begin{theorem}[\cite{kara2020near} Theorem 17]\label{L_uniform_thm}
    If $\beta<\frac{1}{\alpha_{z}}$, 
    where $\alpha_{z}$ is a constant satisfying 
    $\rho_{B L}\left(\eta(\cdot \mid z, u), \eta\left(\cdot \mid z^{\prime}, u\right)\right) \leq \alpha_{z} \rho_{B L}\left(z, z^{\prime}\right)$ for every $u\in \U$ and $z,z^{\prime} \in P(\X)$
    ($\alpha_{z}$ can be chosen as $(3-2 \delta(Q))(1-\tilde{\delta}(\mathcal{T}))$ by \cite{kara2020near} Theorem 7), 
    we have that
    \begin{enumerate}
        \item[(i)]
$$
\sup _z\left|J_\beta^N(z)-J_\beta^*(z)\right| \leq \frac{\left(\alpha_{z}-1\right) \beta+1}{(1-\beta)^2\left(1-\alpha_{z} \beta\right)}\|c\|_{\infty} \bar{L}_{T V}^N
$$
\item[(ii)]
$$
\sup _z\left|J_\beta\left(z, \gamma_N\right)-J_\beta^*(z)\right| \leq \frac{2\left(1+\left(\alpha_{z}-1\right) \beta\right)}{(1-\beta)^3\left(1-\alpha_{z} \beta\right)}\|c\|_{\infty} \bar{L}_{T V}^N
$$
\end{enumerate}
   \end{theorem}

   The upper bound for $\bar{L}_{T V}^N$ provided in Theorem \ref{d} thus is applicable. Unlike \cite[Theorem 3]{kara2021convergence}, it provides a uniform bound for the error term.
       
Therefore, to achieve an $\epsilon$-near optimal policy, it is sufficient to choose $N \in O(\log \frac{1}{\epsilon})$.

\begin{remark}
Under a persistence of excitation of the optimal policy and minorization-majorization assumptions, \cite{cayci2022} demonstrates that, as $N$ increases, the error term, similar to Lemma \ref{L_t_n}, converges to zero geometrically. Unlike Assumption \ref{mixing_kernel_con}, \cite[Condition 1]{cayci2022} (Persistence of excitation of the optimal policy) requires that $\su (\phi^*(.|y,I^N))$ be the same for every $y$ and N-window $I^N$. This implies that the optimal policy must be strictly non-deterministic. We note also that the state space in our setup is not necessarily finite. 
\end{remark}

\section{Examples on Filter Stability Terms Implying Near-Optimality under Expected Filter Stability and Uniform (Sample-Path) Filter Stability}

In this section, we will provide examples that satisfy either Assumption \ref{main_assumption} or Assumption \ref{mixing_kernel_con}. In these examples, as mentioned above, it will be sufficient to have 
$N \in O(\log \frac{1}{\epsilon})$ for achieving an $\epsilon$-near optimal policy.

\subsection{Finite State Space: Expected Filter Stability Bound}
The first example will be for the discrete case.
For the case with finite $\X$, consider 
the discrete metric $d$ 
defined as follows:
$$d\left(x, x^{\prime}\right)= \begin{cases}1 & \text { if } x \neq x^{\prime} \\ 0 & \text { if } x=x^{\prime}.\end{cases}$$
With this choice of metric, the diameter $D$ is equal to $1$.
\begin{example}\label{ex1}
        Let $\X=\{0, 1, 2, 3\}$, $\Y=\{0, 1\}$,
        $\U=\{0, 1\}$, $\epsilon \in (0,1/2)$, and
        let $c$ be any function from $\X \times \U$ to $\mathbb{R}_+$. 
        Now, consider the transition and measurement matrices given by:
        \begin{align*}
        \sT_0:=\mathcal{T}(x_1|x_0,u_0=0)=\left(\begin{array}{cccc}
        1/2 & 1/3 & 1/6 & 0 \\
        0 & 1/2 & 1/6 & 1/3 \\
        1/2 & 1/6 & 0 & 1/3 \\
        1/3 & 1/3 & 1/3 & 0
        \end{array}\right)
        \quad 
        Q=\left(\begin{array}{ccc}
        1-\epsilon & \epsilon\\
        1-\epsilon & \epsilon \\
        \epsilon & 1-\epsilon\\
        \epsilon & 1-\epsilon
        \end{array}\right),
        \end{align*}
        \begin{align*}
            \sT_1=\mathcal{T}(x_1|x_0,u_0=1)=\left(\begin{array}{cccc}
            1/3 & 1/2 & 1/6 & 0 \\
            0 & 1/3 & 1/2 & 1/6 \\
            1/2 & 1/3 & 0 & 1/6 \\
            1/3 & 1/3 & 1/3 & 0
            \end{array}\right)
        \end{align*}
\end{example}
For this example, please note that 
$\delta(Q)=2\epsilon$, and the diameter 
$D$ is equal to 1. We can choose 
$\alpha$ to be 1. And therefore Theorem \ref{J_diff_policy} and Corollary \ref{Corollary_j_diff} are applicable.
Hence, 
\begin{align}
    E_{z_0^{-}}^{\hat{\gamma}}\left[\left|\tilde{J}_\beta^N\left(\hat{z}_0, \tilde{\phi}^N\right)
    -J_\beta^*\left(\hat{z}_0\right)\right| \mid I_0^N\right]  \leq 
    \|c\|_{\infty} (2-\beta)(1-\epsilon)^N. 
\end{align}
Alternatively, if $\delta(Q)$ is close to 1, a more relaxed condition on 
$\alpha$ would suffice. 

\subsection{Uncountable State Space: Expected Filter Stability Bound}
The second example will be for uncountable $\X$.
\begin{example}\label{ContE}
    Consider $\X=[0,1]$, $\U=\{0, 1, \cdots, p\}$ and the cost function $c(x,u)=x-u$. Define the transition kernel $\T(.|x,u)=\bar{N}(x+u,\sigma^2)$.
    Here, $\bar{N}(\mu,\sigma^2)$ denote truncated version 
    of $N(\mu,\sigma^2)$, where the support is restricted to $[0,1] \subset \mathbb{R}$.
    Its probability density function $f$ is given by
    $$f(x ; \mu, \sigma)=\frac{1}{\sigma} 
    \frac{\varphi\left(\frac{x-\mu}{\sigma}\right)}
    {\Phi\left(\frac{1-\mu}{\sigma}\right)-\Phi\left(\frac{0-\mu}{\sigma}\right)}$$
    Here, $\varphi(\cdot)$ is the probability density function 
    of the standard normal distribution and 
    $\Phi(\cdot)$ is its cumulative distribution function.

    Checking Assumption \ref{main_assumption} for this example, we find
    $D=1$.
    For any $0\leq x < y \leq 1$:
    \[
    \frac{\lVert \T(.|y,u)-\T(.|x,u) \rVert _{TV}}{y-x}\leq
    \frac{\sqrt{2}}{\sigma \sqrt{\pi}}.
    \]
    This shows that the transition kernel $\T$ satisfies Assumption \ref{main_assumption}-\ref{regularity} with $\alpha = \frac{\sqrt{2}}{\sigma \sqrt{\pi}}$.
    If we choose $\sigma>\sqrt{2}/\sqrt{\pi}$, 
    then $\alpha$ will be less than $1$.
    For any observation channel, Theorem \ref{J_diff_policy} and Corollary \ref{Corollary_j_diff} are applicable, and thus we obtain:
    \begin{align}
        E_{z_0^{-}}^{\hat{\gamma}}\left[\left|\tilde{J}_\beta^N\left(\hat{z}_0, \tilde{\phi}^N\right)
        -J_\beta^*\left(\hat{z}_0\right)\right| \mid I_0^N\right]  \leq 
        (1 + (p-1) \beta) \left(\frac{\alpha (2-\delta(Q))}{2}  \right)^N. 
    \end{align}
    Here, the term on the right converges to $0$ exponentially fast.

    Additionally, for this example, it is possible to construct an observation channel that also satisfies Assumption \ref{mixing_kernel_con}.
\end{example}

\subsection{Uncountable or Finite State Space: Uniform Sample Path-Wise Filter Stability}

\begin{example}\label{ex3}
        Let $\X=\{0, 1, 2\}$, $\Y=\{0, 1\}$,
        $\U=\{0, 1\}$, $\epsilon \in (0,1/2)$, and
        let $c$ be any function from $X \times U$ to $\mathbb{R}_+$. 
        Now, consider the transition and measurement matrices given by:
        \begin{align*}
        \sT_0=\left(\begin{array}{ccc}
        1/2 & 1/3 & 1/6  \\
        1/3 & 1/2 & 1/6  \\
        1/2 & 1/6 & 1/3 
        \end{array}\right),
        \quad 
        Q=\left(\begin{array}{ccc}
        1-\epsilon & \epsilon\\
        1-\epsilon & \epsilon \\
        \epsilon & 1-\epsilon
        \end{array}\right),
        \quad 
            \sT_1=\left(\begin{array}{cccc}
            1/3 & 1/2 & 1/6 \\
            1/6 & 1/3 & 1/2 \\
            2/3 & 1/6 & 1/6
            \end{array}\right)
        \end{align*}
\end{example}
In this example, note that $Q(y|x) \geq \epsilon$ for every $x \in \X$ and $y \in \Y$. The matrices $\sT_0$ and $\sT_1$ are mixing kernels with mixing coefficients $\sqrt{1/3}$ and $1/2$, respectively, which satisfy Assumption \ref{mixing_kernel_con}. Therefore, by Theorem \ref{d}, we obtain:
\begin{align}
\bar{L}_{TV}^N \leq c^{N-1} K.
\end{align}
Here, $K = \frac{2}{\log 3} \sup h(Z_1, Z_1^*)$ and $c = \sup_{u \in \U} \frac{1-\epsilon_{u}^2\epsilon}{1+\epsilon_{u}^2\epsilon} = \frac{3-\epsilon}{3+\epsilon}$.
As a result, Theorem \ref{L_uniform_thm} is applicable, and the error term converges to $0$ geometrically fast as $N$ approaches infinity. Note that for sufficiently small values of $\epsilon > 0$, the bounds leading to Theorem \ref{J_diff_policy} would not have been applicable under the Dobrushin coefficient based bounds, but would be so under the Hilbert metric.


\section{Implications on Learning}
In the results obtained above, we assume $\mathbb{U}$ is
finite. However, if $\mathbb{U}$ is compact, similar 
results can be obtained through quantization. 
As demonstrated in \cite[Theorem 3.2]{SaLiYuSpringer}, 
if the transition kernel is weakly continuous and the 
cost function is bounded and continuous, then the optimal 
policy for a quantized action model for the discounted cost
is also near-optimal for the original model. 
This allows us to consider $\mathbb{U}$ as finite to 
identify a near-optimal policy.

We will now describe the process for identifying all 
optimal policies using Q-learning for the approximate 
finite belief MDP in the context of discounted cost. 
By fixing the posterior distribution $z^*$, we track 
both observations and actions. We assume tracking of 
the last $N+1$ observations and the last $N$ 
control actions after at least $N+1$ time steps. 
Thus, at time $n$, we monitor the information 
variables $I_n^N$.

The Q-value iteration is constructed using 
these information variables. For these new 
approximate states, we follow the standard 
Q-learning algorithm. For any 
$I \in \mathbb{Y}^{N+1} \times \mathbb{U}^N$ and 
$u \in \mathbb{U}$, the Q-value update is as follows:
\begin{align}\label{Q_l_alg_finite}
Q_{t+1}(I, u)=\left(1-\alpha_t(I, u)\right) 
Q_t(I, u)+\alpha_t(I, u)\left(\hat{c}^N(I, u)+
\beta \min _v Q_t\left(I_1^t, v\right)\right),
\end{align}
where $I_1^t=\left\{Y_{t+1}, y_t, \ldots, y_{t-N+1},
 u_t, \ldots, u_{t-N+1}\right\}$.
Exploration policies are employed, where control 
actions are chosen independently at random. 
At time $t$, the action $u_t$ is 
selected with probability $\sigma_i$ for 
each $u_i \in \mathbb{U}$, where $\sigma_i > 0$ for all $i$.
\begin{assumption}
    [\cite{kara2021convergence} Assumption 4.1]
    \label{Q_assumption_finite}
    \noindent
\begin{enumerate}
\item $\alpha_t(I, u)=0$ unless 
$\left(I_t, u_t\right)=(I, u)$. In other cases,
$$
\alpha_t(I, u)=\frac{1}{1+\sum_{k=0}^t \1_{\left\{I_k=I, u_k=u\right\}}}
$$
\item\label{Harris} Under the stationary (memoryless or finite memory exploration)
policy, say $\gamma$, the true state process, 
$\left\{X_t\right\}_t$, is positive Harris 
recurrent and in particular admits a unique 
invariant measure $\pi_\gamma^*$.
\item Furthermore, we have that $Q(y | x) > 0$ for every $x \in \mathbb{X}$, and during the exploration phase every 
$(I, u)$ pair is visited infinitely often.
\end{enumerate}
\end{assumption}
\begin{theorem}
    [\cite{kara2021convergence}  
    Theorem 4, Corollary 1]
Suppose Assumption \ref{Q_assumption_finite} holds.
 Then, the following statements are true:
\begin{itemize}
\item The algorithm given in (\ref{Q_l_alg_finite}) 
converges almost surely to $Q^*$ which satisfies 
fixed point equation.
\item A stationary policy $\gamma^N$ that
 satisfies $\gamma^N(I) \in \argmin_u Q^*(I, u)$ for all $I$, is an optimal policy, i.e.,
if we assume that the controller starts using $\gamma^N$ at time $t=N$ (after observing at least $N$ information variables), 
then denoting the prior distribution of $X_N$ by $z_N^{-}$ conditioned on the first $N$ step information variables, we have as $N \rightarrow \infty$
\begin{align}
E\left[J_\beta\left(z_N^{-}, \gamma^N\right)-J_\beta^*\left(z_N^{-}\right) \mid I_0^N\right] \leq \frac{2\|c\|_{\infty}}{(1-\beta)} \sum_{t=0}^{\infty} \beta^t L_t .
\end{align}
(The refined upper bound for this error term is provided in Corollary \ref{Corollary_j_diff} and Corollary \ref{cor_kara_y}.)
In particular, if the POMDP is such that the filter is
uniformly stable over priors in expectation
under total variation, meaning
$L^N_t \rightarrow 0$ as $N \rightarrow \infty$, then
\begin{align}
E\left[J_\beta\left(z_N^{-}, \gamma^N\right)-J_\beta^*\left(z_N^{-}\right) \mid I_0^N\right] \rightarrow 0 .
\end{align}
\end{itemize} 
\end{theorem}
\begin{remark}
We know that under Assumption
\ref{mixing_kernel_con}, $L^N_t$ converges to
zero geometrically fast. Additionally,
using Theorem \ref{J_diff_policy}, the error term can be bounded similarly using $\bar{L}^N_t$. The convergence of $\bar{L}^N_t$ to zero implies that the error term also converges to zero. From Corollary \ref{Corl}, we know that this condition is satisfied under Assumption \ref{main_assumption}.
\end{remark}
    This method allows for obtaining
    near-optimal policies for POMDPs under the discounted cost criterion.
    The same method can also be applied to achieve near-optimal policies
    for the average cost criterion, as discussed in \cite{DeKaYu2024}.
It is important to note that Q-learning provides the best $N$-window policy for the process starting after $N$ steps. If we want to account for the first $N$ steps as well, we need to solve a finite-horizon optimal problem with a terminal cost, where the terminal cost is determined using the algorithm described above.

\section{Simulations}

In this section, we outline the algorithm used to determine the approximate belief MDP and the finite window policy. We also present the performance of the finite window policies and the value function error for the approximate belief MDP concerning the window size.

\subsection{Algorithm Summary}
The algorithm is summarized as follows:
\begin{enumerate}
    \item Determine a $z^* \in \mathcal{P}(\mathbb{X})$ such that it assigns a positive measure over the state space $\mathbb{X}$.
    \item Following Section \ref{Alt_finite_Mdp}, construct the finite belief space $ \hat{\mathcal{Z}}_{z^*}^N$ by taking the Bayesian update of $z^*$ for all possible realizations of the form $\{y_0, \ldots, y_N, u_0, \ldots, u_{N-1}\}$. This results in an approximate finite belief space with size $|\mathbb{Y}|^{N+1} \times |\mathbb{U}|^N$.
    \item For the obtained finite models, use Q-learning to find the optimal policy and optimal values.
\end{enumerate}

\subsection{Example: Machine Repair Problem}
Consider the machine repair problem where $\mathbb{X}, \mathbb{Y}, \mathbb{U} = \{0, 1\}$ with:
\[
x_t = \begin{cases} 
1 & \text{machine is working at time } t \\
0 & \text{machine is not working at time } t 
\end{cases}
\quad
u_t = \begin{cases}
1 & \text{machine is being repaired at time } t \\
0 & \text{machine is not being repaired at time } t 
\end{cases}
\]
The probability that the repair was successful given that initially the machine was not working is given by $\kappa$:
$
\operatorname{Pr}(x_{t+1}=1 \mid x_t=0, u_t=1) = \kappa.
$
The probability that the machine breaks down while in a working state is given by $\theta$:
$
\operatorname{Pr}(x_t=0 \mid x_t=1, u_t=0) = \theta.
$
The probability that the channel gives an incorrect measurement is given by $\epsilon$:
$
\operatorname{Pr}(y_t=1 \mid x_t=0) = \operatorname{Pr}(y_t=0 \mid x_t=1) = \epsilon.
$
The one-stage cost function is:
\[
c(x, u) = \begin{cases}
R + E & x = 0, u = 1 \\
E & x = 0, u = 0 \\
0 & x = 1, u = 0 \\
R & x = 1, u = 1
\end{cases}
\]
where $R$ is the cost of repair and $E$ is the cost incurred by a broken machine.

We study the example with a discount factor $\beta = 0.8$, and $R = 2, E = 1$, presenting three different results by changing other parameters. For the exploring policy, we use a random policy such that $\operatorname{Pr}(\gamma(x)=0)=\frac{1}{2}$ and $\operatorname{Pr}(\gamma(x)=1)=\frac{1}{2}$ for all $x$.
  Under this policy, $x_t$ admits a stationary policy $z^*(\cdot)= \frac{\theta}{\theta+\kappa} \delta_0(\cdot)$ $+ \frac{\kappa}{\theta+\kappa} \delta_1(\cdot)$.

First, we will check for which values Assumption \ref{main_assumption} holds for this example.
Assumption \ref{main_assumption}-\ref{compactness} holds because $\mathbb{X}$ is finite, hence compact. We will use the discrete metric on $\mathbb{X}$, defined as $d(x,y) = 0$ if $x = y$ and $d(x,y) = 1$ otherwise. Therefore, $D = \sup_{x,y \in \mathbb{X}} d(x,y) = 1$. 
Assumption \ref{main_assumption}-\ref{totalvar} and \ref{main_assumption}-\ref{bdd} are satisfied since both $\mathbb{X}$ and $\mathbb{U}$ are finite.
Assumption \ref{main_assumption}-\ref{regularity} holds for $\alpha = 2 \max(1 - \theta, 1 - \kappa)$.
Assumption \ref{main_assumption}-\ref{CostLipschitz} is satisfied with $K_1 = E = 1$.
$\delta(Q)=2\epsilon$, and from Theorem \ref{J_diff_policy} and Corollary \ref{Corl}, we derive:
\begin{align*}
    E_{z_0^{-}}^{\hat{\gamma}}\left[\left|\tilde{J}_\beta^N\left(\hat{z}_0, \tilde{\phi}^N\right)
    -J_\beta^*\left(\hat{z}_0\right)\right| \mid I_0^N\right]  \leq 2 \left(E + \alpha \beta \frac{R+E}{1-\beta} \right) \sum_{t=0}^{\infty} \beta^t \bar{L}^N_t
\end{align*}
and
$
    L^N \leq \frac{1}{2} \left(\alpha (1-\epsilon)  \right)^N,
$
where
$$
L^N:=\sup _{z \in \mathcal{P}(\mathbb{X})} \sup _{\hat{\gamma} \in \hat{\Gamma}} E_z^{\hat{\gamma}}\left[W_1\left(P^z\left(X_N \in \cdot \mid Y_{[0, N]}, U_{[0, N-1]}\right), P^{z^*}\left(X_N \in \cdot \mid Y_{[0, N]}, U_{[0, N-1]}\right)\right)\right]
$$
which is an upper bound for $\bar{L}^N_t$.

In the simulation, when taking the supremum for $L^N$, instead of considering control Markovian policies over $\hat{\Gamma}$, we considered all possible action sequences $u_{[0, N-1]}$. This provides an upper bound for $L^N$. To estimate $J_\beta^*(\hat{z}_0)$, we simply use the smallest value of $\tilde{J}_\beta^N\left(\hat{z}_0, \tilde{\phi}^N\right)$ among the different $N$ values. Furthermore, we scale the $L^N$ and Dobrushin values to show the rate dependence between $\tilde{J}_\beta^N\left(\hat{z}_0, \tilde{\phi}^N\right)-J_\beta^*(\hat{z}_0)$ and $L^N$ more clearly.
\subsection{Case Studies}
\subsubsection{First Case}
For the first case, we take $\epsilon = 0.3, \kappa = 0.3, \theta = 0.3$ so that $ \bar{L}_{t}^N \leq \frac{1}{2} \left(0.98  \right)^N$. Figure \ref{fig:gr1} shows the error $\tilde{J}_\beta^N\left(\hat{z}_0, \tilde{\phi}^N\right)-J_\beta^*(\hat{z}_0)$, $L^N$ and $\frac{1}{2} \left(\alpha (1-\epsilon)  \right)^N$.

\begin{figure}[htbp]
    \centering
    \includegraphics[width=0.5\textwidth]{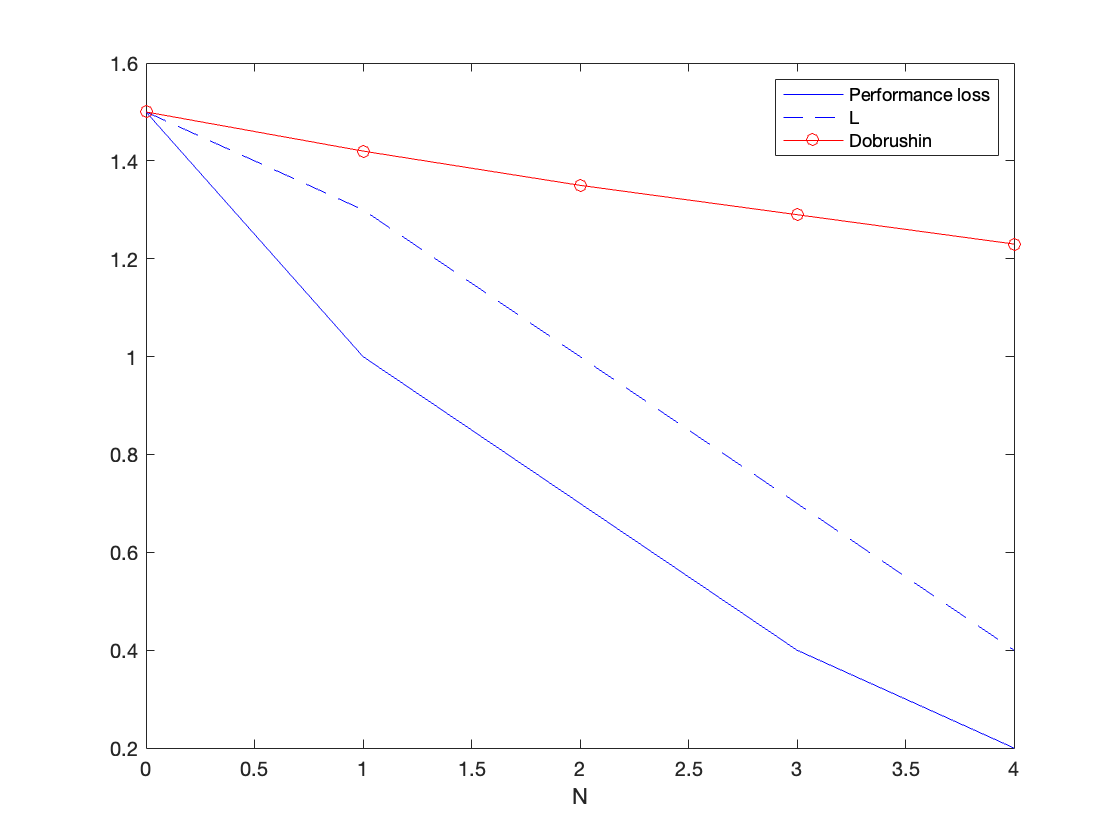} 
    \caption{$\epsilon = 0.3, \kappa = 0.3, \theta = 0.3$}
    \label{fig:gr1}
\end{figure}

\subsubsection{Second Case}
For the second case, we take $\epsilon = 0.2, \kappa = 0.4, \theta = 0.4$ so that $ \bar{L}_{t}^N \leq \frac{1}{2} \left(0.96  \right)^N$. Figure \ref{fig:gr2} shows the error $\tilde{J}_\beta^N\left(\hat{z}_0, \tilde{\phi}^N\right)-J_\beta^*(\hat{z}_0)$, $L^N$ and $\frac{1}{2} \left(\alpha (1-\epsilon)  \right)^N$.
\begin{figure}[htbp]
    \centering
    \includegraphics[width=0.5\textwidth]{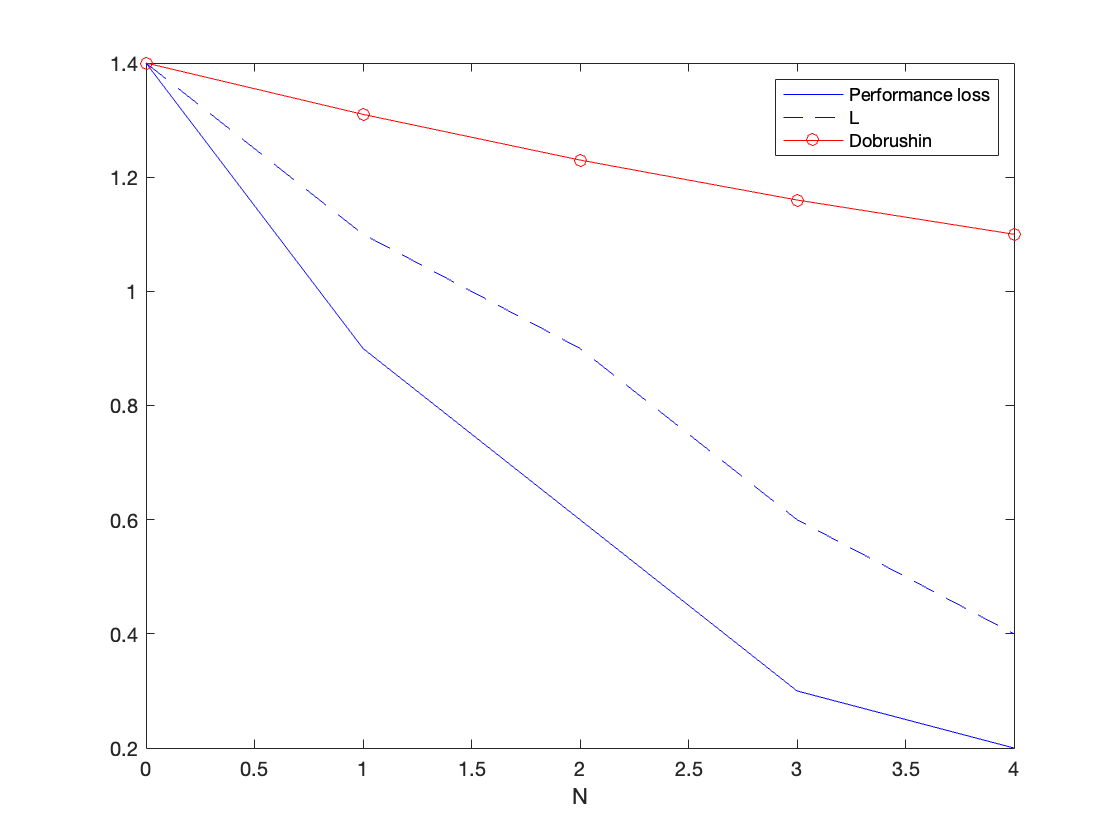} 
    \caption{ $\epsilon = 0.2, \kappa = 0.4, \theta = 0.4$}
    \label{fig:gr2}
\end{figure}

\section{Conclusion}
In this paper, we provided refined bounds on the near-optimality of finite window policies in partially observable Markov decision processes (POMDPs), significantly generalizing prior work, to include Wasserstein filter stability in expectation and uniform sample path-wise filter stability in total variation, under complementary conditions and explicit examples. We demonstrated that these refined bounds offer exponential decay of the error term with respect to the window size $N$, allowing for more precise control over approximation errors in POMDP solutions. These results contribute to rigorous approximations and reinforcement learning implementation to find near-optimal policies in POMDPs.

\bibliographystyle{plain}
\bibliography{EmreBibliography.bib}

\begin{thebibliography}{10}

\bibitem{Bog07}
V.I. Bogachev.
\newblock {\em Measure Theory: Volume II}.
\newblock Springer, 2007.

\bibitem{cayci2022}
S.~Cayci, N.~He, and R.~Srikant.
\newblock Finite-time analysis of entropy-regularized neural natural
  actor-critic algorithm.
\newblock {\em arXiv preprint arXiv:2206.00833}, 2022.

\bibitem{CrDo02}
D.~Crisan and A.~Doucet.
\newblock A survey of convergence results on particle filtering methods for
  practitioners.
\newblock {\em IEEE Transactions on Signal Processing}, 50(3):736--746, 2002.

\bibitem{DeKaYu2024}
Y.~E. Demirci, A.~D. Kara, and S.~Y{\"u}ksel.
\newblock Average cost optimality of partially observed mdps: Contraction of
  non-linear filters and existence of optimal solutions.
\newblock {\em arXiv preprint arXiv:2312.14111, 2023}, 2023.

\bibitem{demirci2023geometric}
Y.~E. Demirci and S.~Y{\"u}ksel.
\newblock Geometric ergodicity, unique ergodicity and wasserstein continuity of
  non-linear filters with compact state space.
\newblock {\em arXiv preprint arXiv:2307.15764}, 2023.

\bibitem{dobrushin1956central}
R.L. Dobrushin.
\newblock Central limit theorem for nonstationary {M}arkov chains. i.
\newblock {\em Theory of Probability \& Its Applications}, 1(1):65--80, 1956.

\bibitem{FeKaZg14}
E.A. Feinberg, P.O. Kasyanov, and M.Z. Zgurovsky.
\newblock Partially observable total-cost {M}arkov decision process with weakly
  continuous transition probabilities.
\newblock {\em Mathematics of Operations Research}, 41(2):656--681, 2016.

\bibitem{feinberg2022markov}
E.A. Feinberg, P.O. Kasyanov, and M.Z. Zgurovsky.
\newblock Markov decision processes with incomplete information and semiuniform
  feller transition probabilities.
\newblock {\em SIAM Journal on Control and Optimization}, 60(4):2488--2513,
  2022.

\bibitem{Feinberg2023}
Eugene~A. Feinberg and Pavlo~O. Kasyanov.
\newblock Equivalent conditions for weak continuity of nonlinear filters.
\newblock {\em Systems \& Control Letters}, 173:105458, 2023.

\bibitem{le2004stability}
F.~Le Gland and N.~Oudjane.
\newblock Stability and uniform approximation of nonlinear filters using the
  {H}ilbert metric and application to particle filters.
\newblock {\em The Annals of Applied Probability}, 14(1):144--187, 2004.

\bibitem{HernandezLermaMCP}
O.~Hern{\'a}ndez-Lerma and J.~B. Lasserre.
\newblock {\em Discrete-Time {M}arkov Control Processes: Basic Optimality
  Criteria}.
\newblock Springer, 1996.

\bibitem{himmelberg1976optimal}
C.~J. Himmelberg, T.~Parthasarathy, and F.~S.~Van Vleck.
\newblock Optimal plans for dynamic programming problems.
\newblock {\em Mathematics of Operations Research}, 1(4):390--394, 1976.

\bibitem{white1994finite}
C.~C.~White III and W.~T. Scherer.
\newblock Finite-memory suboptimal design for partially observed markov
  decision processes.
\newblock {\em Operations Research}, 42(3):439--455, 1994.

\bibitem{KSYWeakFellerSysCont}
A.D Kara, N.~Saldi, and S.~Y\"uksel.
\newblock Weak {F}eller property of non-linear filters.
\newblock {\em Systems \& Control Letters}, 134:104--512, 2019.

\bibitem{KSYContQLearning}
A.D Kara, N.~Saldi, and S.~Y\"uksel.
\newblock Q-learning for mdps with general spaces: Convergence and near
  optimality via quantization under weak continuity.
\newblock {\em Journal of Machine Learning Research, 2023 (arXiv:2111.06781)},
  2023.

\bibitem{kara2020near}
A.D Kara and S.~Y\"uksel.
\newblock Near optimality of finite memory feedback policies in partially
  observed markov decision processes.
\newblock {\em Journal of Machine Learning Research}, 23(11):1--46, 2022.

\bibitem{kara2021convergence}
A.D Kara and S.~Y\"uksel.
\newblock Convergence of finite memory {Q}-learning for {POMDP}s and near
  optimality of learned policies under filter stability.
\newblock {\em Mathematics of Operations Research (also arXiv:2103.12158)},
  2023.

\bibitem{karayukselNonMarkovian}
A.D. Kara and S.~Y\"uksel.
\newblock Q-learning for stochastic control under general information
  structures and non-{M}arkovian environments.
\newblock {\em Transactions on Machine Learning Research (arXiv:2311.00123)},
  2024.

\bibitem{pineau2006anytime}
J.~Pineau, G.~Gordon, and S.~Thrun.
\newblock Anytime point-based approximations for large pomdps.
\newblock {\em Journal of Artificial Intelligence Research}, 27:335--380, 2006.

\bibitem{porta2006point}
J.~M. Porta, N.~Vlassis, M.~T.~J. Spaan, and P.~Poupart.
\newblock Point-based value iteration for continuous pomdps.
\newblock {\em Journal of Machine Learning Research}, 7(Nov):2329--2367, 2006.

\bibitem{Rhe74}
D.~Rhenius.
\newblock Incomplete information in {M}arkovian decision models.
\newblock {\em Ann. Statist.}, 2:1327--1334, 1974.

\bibitem{SaLiYuSpringer}
N.~Saldi, T.~Linder, and S.~Y\"uksel.
\newblock {\em Finite Approximations in Discrete-Time Stochastic Control:
  Quantized Models and Asymptotic Optimality}.
\newblock Springer, Cham, 2018.

\bibitem{SaYuLi15c}
N.~Saldi, S.~Y{\"u}ksel, and T.~Linder.
\newblock On the asymptotic optimality of finite approximations to markov
  decision processes with borel spaces.
\newblock {\em Mathematics of Operations Research}, 42(4):945--978, 2017.

\bibitem{SYLTAC2017POMDP}
N.~Saldi, S.~Y\"uksel, and T.~Linder.
\newblock Finite model approximations for partially observed markov decision
  processes with discounted cost.
\newblock {\em IEEE Transactions on Automatic Control}, 65, 2020.

\bibitem{singh1994learning}
S.P. Singh, T.~Jaakkola, and M.I. Jordan.
\newblock Learning without state-estimation in partially observable markovian
  decision processes.
\newblock In {\em Machine Learning Proceedings 1994}, pages 284--292. Elsevier,
  1994.

\bibitem{SinhaMahajan2024}
A.~Sinha and A.~Mahajan.
\newblock Agent-state based policies in pomdps: Beyond belief-state mdps.
\newblock In {\em 2024 IEEE Conference on Decision and Control, Tutorial
  Paper}. IEEE, 2024.

\bibitem{smith2012point}
T.~Smith and R.~Simmons.
\newblock Point-based pomdp algorithms: Improved analysis and implementation.
\newblock {\em arXiv preprint arXiv:1207.1412}, 2012.

\bibitem{villani2008optimal}
C.~Villani.
\newblock {\em Optimal transport: old and new}.
\newblock Springer, 2008.

\bibitem{spaan2005perseus}
N.~Vlassis and M.~T.~J. Spaan.
\newblock Perseus: Randomized point-based value iteration for pomdps.
\newblock {\em Journal of artificial intelligence research}, 24:195--220, 2005.

\bibitem{yu2008near}
H.~Yu and D.~P. Bertsekas.
\newblock On near optimality of the set of finite-state controllers for average
  cost pomdp.
\newblock {\em Mathematics of Operations Research}, 33(1):1--11, 2008.

\bibitem{Yus76}
A.A. Yushkevich.
\newblock Reduction of a controlled {M}arkov model with incomplete data to a
  problem with complete information in the case of {B}orel state and control
  spaces.
\newblock {\em Theory Prob. Appl.}, 21:153--158, 1976.

\bibitem{zhou2008density}
E.~Zhou, M.~C. Fu, and S.~I. Marcus.
\newblock A density projection approach to dimension reduction for
  continuous-state {P}{O}{M}{D}{P}s.
\newblock In {\em Decision and Control, 2008. CDC 2008. 47th IEEE Conference
  on}, pages 5576--5581, 2008.

\bibitem{zhou2010solving}
E.~Zhou, M.~C. Fu, and S.~I. Marcus.
\newblock Solving continuous-state {P}{O}{M}{D}{P}s via density projection.
\newblock {\em IEEE Transactions on Automatic Control}, 55(5):1101 -- 1116,
  2010.

\end{thebibliography}

\end{document}